\documentclass{article}
\usepackage{amsmath}
\usepackage{amssymb,mathrsfs}
\usepackage{amsthm}
\usepackage{bm}
\usepackage{fixmath}
\usepackage[utf8]{inputenc}
\usepackage{graphicx}
\usepackage{xcolor}
\usepackage{mathtools}
\usepackage{hyperref}
\hypersetup{
  colorlinks=true,
  linkcolor=black,
  citecolor=black
}
\usepackage[capitalize]{cleveref}

\usepackage{booktabs}
\usepackage{multirow}

\def\R{\mathbb R}
\newcommand{\nfrac}[2]{{#1}/{#2}}
\DeclarePairedDelimiter\myset\{\}
\newcommand{\set}[2][*]{\expandafter\myset#1{#2}}
\newcommand{\vv}[2]{\begin{pmatrix}#1\\#2\end{pmatrix}}
\newcommand{\vvv}[3]{\begin{pmatrix}#1\\#2\\#3\end{pmatrix}}

\newcommand{\mesh}[1][h]{\mathcal{T}_{#1}}
\newcommand{\coarsemesh}{\mesh[H]}
\def\cell{T}
\def\face{F}
\def\boundary{B}
\def\setoffaces{\Gamma_I}
\def\setofboundaries{\Gamma_B}
\def\diam{\operatorname{diam}}

\def\jump#1{[\![#1]\!]}
\def\mean#1{\{\!\!\{#1\}\!\!\}}
\def\eigenvalue{\lambda}
\newcommand{\Ltwoprojection}[1][h]{\Pi^{#1}}
\newcommand{\interpolant}[1][h]{I^{#1}}

\def\ph{p_h}

\def\gh{g_h}

\def\bu{\boldsymbol u}
\def\bv{\boldsymbol v}
\def\bw{\boldsymbol w}
\def\bx{\boldsymbol x}
\def\by{\boldsymbol y}
\def\bpsi{\boldsymbol \psi}
\def\bphi{\boldsymbol \phi}
\def\bpi{\boldsymbol \pi}
\def\bbf{\boldsymbol f}
\def\bn{\boldsymbol n}

\def\btau{\bm \tau}

\def\kerDform{\ker(\mathscr D)}
\newcommand{\Wh}[1][]{W_h^{#1}}
\newcommand{\Wj}[1][]{W_j^{#1}}

\def\mixedP{\mathbb P}
\def\systemoperator{\mathbb{A}}
\def\Pad{\mathbb{P}_{\text{ad}}}
\def\Pmu{\mathbb{P}_{\text{mu}}}

\newcommand{\Emu}{\mathbb{E}_{\text{mu}}}

\def\Mmu{\mathbb{M}_{\text{mu}}}

\def\singularP{P}
\def\singularPad{P_{\text{ad}}}
\def\singularPmu{P_{\text{mu}}}
\def\singularEmu{E_{\text{mu}}}
\def\singularoperator{A}

\def\storage{c_s}
\def\willis{\alpha}
\def\permeability{\bm K}
\def\scaledlambda{\hat\lambda}
\def\scaledpermeability{\hat\permeability}
\def\scaledkappa{\hat\kappa}
\def\scaledstorage{\hat c_s}

\def\scaledRmax{\hat R_{\max}}
\def\scaledRmin{\hat R_{\min}}

\newcommand{\constK}{c_{\scaledpermeability}}
\def\relaxation{\omega}
\def\relaxationconstant{\relaxation_0}
\def\localstabconstant{C_1}
\def\CSconstant{\epsilon_{ij}}
\def\spectralradius{\rho(\mathcal{E})}
\def\constantinfsupu{\gamma_u}
\def\constantinfsupv{\gamma_v}

\def\strain#1{\bm \varepsilon (#1)}

\def\div#1{\operatorname{div} #1}
\def\curl{\operatorname{\bf curl}}

\def\Hdiv{H^{\text{div}}}
\def\Hcurl{H^{\curl}}

\def\scal(#1){\left(#1\right)}
\makeatletter
\def\ltwoscal{\@ifnextchar[{\@ltwoscal}{\@ltwoscal[\Omega]}}
\def\@ltwoscal[#1](#2){\scal(#2)_{#1}}
\makeatother

\def\abs#1{\left|#1\right|}
\def\norm#1{\left\|#1\right\|}
\newcommand{\ltwonorm}[2][\Omega]{\norm{#2}_{#1}}
\newcommand{\linftynorm}[2][]{\norm{#2}_{L^\infty#1}}
\newcommand{\Honenorm}[1]{\norm{#1}_{1}}
\newcommand{\Hdivnorm}[1]{\norm{#1}_{\Hdiv}}
\newcommand{\dgnormone}[2][h]{\norm{#2}_{1,#1}}
\newcommand{\dgnormtwo}[2][h]{\norm{#2}_{2,#1}}
\newcommand{\Whnorm}[1]{\norm{#1}_{\Wh}}

\def\aform(#1){a\!\left(#1\right)}
\def\ajform(#1){a_j\!\left(#1\right)}
\def\eeform(#1){e\!\left(#1\right)}
\def\eejform(#1){e_j\!\left(#1\right)}
\def\eehform(#1){e_h\!\left(#1\right)}
\def\eeHform(#1){e_H\!\left(#1\right)}
\def\ddform(#1){d\!\left(#1\right)}
\def\kform(#1){k\!\left(#1\right)}
\def\divform(#1){b\!\left(#1\right)}
\def\mixedform(#1){\mathcal A_h\!\left(#1\right)}
\def\coarsemixedform(#1){\mathcal A_H\!\left(#1\right)}
\def\mixedrhs(#1){\mathcal F_h\!\left(#1\right)}
\makeatletter
\def\locmixedform{\@ifnextchar[{\@locmixedform}{\@locmixedform[j]}}
\def\@locmixedform[#1](#2){\mathcal A_{#1}\!\left(#2\right)}
\makeatother
\def\singularform(#1){\mathscr A_h\!\left(#1\right)}
\def\singularrhs(#1){\mathscr F_h\!\left(#1\right)}
\def\coarsesingularform(#1){\mathscr A_H\!\left(#1\right)}
\makeatletter
\def\locsingularform{\@ifnextchar[{\@locsingularform}{\@locsingularform[j]}}
\def\@locsingularform[#1](#2){\mathscr A_{#1}\!\left(#2\right)}
\makeatother
\def\Ehform(#1){\mathscr E_h\!\left(#1\right)}
\def\Dform(#1){\mathscr D\!\left(#1\right)}

\newtheorem{theorem}{Theorem}[section]
\newtheorem{lemma}[theorem]{Lemma}
\newtheorem{assumption}[theorem]{Assumption}
\newtheorem{proposition}[theorem]{Proposition}
\newtheorem{remark}[theorem]{Remark}

\makeatletter
\if@cref@capitalise
\crefname{assumption}{Assumption}{Assumptions}
\else
\crefname{assumption}{assumption}{assumptions}
\fi
\makeatother
\Crefname{assumption}{Assumption}{Assumptions}

\title{Monolithic two-level Schwarz preconditioner for Biot's consolidation model in two space dimensions}
\author{Stefan Meggendorfer\thanks{Institute for Modelling Hydraulic and Environmental Systems (LH2), University of Stuttgart, Germany}
\and
Guido Kanschat\thanks{Interdisciplinary Center for Scientific Computing (IWR), Heidelberg University, Germany}
\and
Johannes Kraus\thanks{Faculty of Mathematics, University of Duisburg-Essen, Germany}}

\begin{document}
\maketitle

\begin{abstract}
This paper addresses the construction and analysis of a class of domain decomposition methods
for the iterative solution of the quasi-static Biot problem in three-field formulation. 
The considered discrete model arises from time discretization by the implicit Euler method and space discretization by a family of strongly mass-conserving methods exploiting $\Hdiv$-conforming approximations of the solid displacement and
fluid flux fields.
For the resulting saddle-point problem, we construct monolithic overlapping domain decomposition (DD) methods whose analysis relies on a transformation into an equivalent symmetric positive definite system and on stable decompositions of the involved finite element spaces under proper problem-dependent norms.
Numerical results on two-dimensional test problems are in accordance with the provided theoretical uniform convergence estimates for the two-level multiplicative Schwarz method.
\end{abstract}

\noindent\textbf{MSC-Classes:} 65M12, 65M55, 65F10
\\[1ex]

\section{Introduction}

Mathematical models describing the interaction of fluid flow with the mechanical deformation of porous media
go back to the pioneering works by Karl von Terzaghi and Maurice Anthony Biot in the first half of the last
century~\cite{Terzaghi1925erdbaumechanic,Biot1941general,Terzaghi1943theoretical,Biot1955theory}.
Important application areas include, but are not limited to, petroleum and reservoir engineering, CO2 sequestration
and the mechanics of biological tissues. Partial saturation, the incorporation of two- or multi-phase flow, thermal
processes, and/or large deformations, and/or the consideration of multiple fluid networks lead to even more complex,
sometimes highly nonlinear mathematical models that require physically correct (structure-preserving) discretization
and robust iterative solution methods. 

The well-posedness of the basic Biot model of consolidation on the continuous and discrete levels has
been established in the works of Zenisek and Showalter using semi-group theory and Galerkin discretization
methods~\cite{Zenisek1984existence,Zenisek1984finite,Showalter00diffusion}.
The simplest linear model uses a two-field formulation of Biot's equations for the solid displacement  $\bu$ and pore
pressure $p$. Popular time discretizations of the quasi-static differential algebraic system are the backward Euler
or high-order strongly stable time integration methods, see, e.g.,~\cite{AxelssonBlahetaKohut15}.
A stability and convergence analysis of inf-sup stable finite element approximations of this two-field model has
first been presented in~\cite{MuradLoula1992improved,MuradLoula1994stability} in the setting of implicit Euler time stepping.
Discretizations based on three-field-formulation have originally been proposed
in~\cite{PhillipsWheeler2007coupling1,PhillipsWheeler2007coupling2} where standard continuous Galerkin
approximations of the displacement variable $\bu$ are coupled to a mixed finite-element method for the pore pressure $p$
based on introducing the flux variable $\bv$ to the system, which obeys the Darcy law. This approach has also been extended
to discontinuous Galerkin approximations of $\bu$ in~\cite{PhillipsWheeler2008coupling} and other nonconforming approximations,
e.g., using modified rotated bilinear elements~\cite{Yi2013coupling}, or Crouzeix-Raviart elements in~\cite{Hu2017nonconforming}.
Lately, the conservation of mass was identified essential for parameter robustness.
In~\cite{HongKraus18,KanschatRiviere18}, families of strongly mass conserving discretizations based on the
$\Hdiv$-conforming discontinuous Galerkin discretization of $\bu$ have been suggested and proven stable independently
of model and discretization parameters (Lam\'{e} parameters, permeability and Biot-Willis parameter, storage coefficient, time step
and mesh size), see, e.g.~\cite{HongKraus18,HongKraus19}, for time-dependent error estimates, see~\cite{KanschatRiviere18}.
More recently, mass-conservation was achieved by finite elements based on enrichment in~\cite{GasparRodrigoHu19,LeeYi22}
and by mixed virtual finite elements in~\cite{WangCai22}.

A discretization of a four-field formulation of Biot's model has been proposed in~\cite{Yi2014convergence}. The method uses the
(effective) stress tensor $\bm \sigma$, the displacement $\bu$, the fluid flux $\bv$, and the pore pressure $p$ as unknowns and
couples two standard mixed finite element methods, one for the mechanics and one for the flow subproblem, the former of which
is based on the Hellinger-Reissner formulation. Optimal a priori error estimates were shown for both semidiscrete and fully discrete
problems when the flux unknown is approximated in the Raviart-Thomas space and the stress tensor in the Arnold-Winther space.
The error analysis of this coupled mixed method has been complemented in~\cite{Lee2016robust} by estimates in $L^\infty$-norm
in time and $L^2$-norm in space that are robust with respect to the Lam\'{e} parameters and do not require strict positivity of the
constrained storage coefficient $\storage$.

When it comes to the solution of the algebraic problems arising from discretization of Biot's problem, preconditioning techniques
have been developed for various formulations and discretizations, in particular, in the framework of norm- or field-of-values-equivalent
operator preconditioners, see,
e.g.~\cite{boffi2016nonconforming,adler2019robust} for two-field,
\cite{oyarzua2016locking, Hu2017nonconforming,Lee2016parameter,HongKraus18,Lee2019mixed,HongKraus19,Hong2022robust,Kraus2023hybridized} for three-field, and~\cite{Lee2016robust,Baerland2017weakly,BoonKuchtaMardal21}
for four-field formulations. 
Here, \cite{HongKraus19,Lee2019mixed,Kraus2023hybridized} focus on multiple network extensions of Biot's model and \cite{Hong2022robust} on a generalization to Biot-Brinkman equations. 
The key task in the development of such methods is the
stability analysis of the underlying saddle-point problems, which can be performed in the framework presented in~\cite{Hong2021new}.

By a decomposition into physical subsystems, the coupled static problem can be solved implicitly, using a loose or explicit coupling,
or an iterative coupling. The latter, often provides an attractive alternative in terms of achieving high accuracy at reasonable computational
cost. The most popular procedures in this category are the undrained split, the fixed-stress split, the drained split and the fixed-strain split
iterative methods, which also generate preconditioners for Krylov methods such as the GMRES method. As shown in~\cite{Kim2011stability},
in contrast to the drained split and the fixed-strain split methods, the undrained split and fixed-stress split methods are unconditionally stable. 
The first convergence analysis of the latter methods has been presented in~\cite{Mikelic2013convergence} for the quasi-static Biot system.
More recent works on fixed-stress type split methods have focused on the parameter-robustness of the scheme and the optimization of the stabilization/acceleration parameter~\cite{Storvik_2019optimization,HongEtAl2020fixed-stress}, extensions to heterogeneous poroelastic media~\cite{AlmaniKumarWheeler23}, and generalizations to nonlinear poroelasticity models~\cite{kraus2023fixedstress}. 
For a parameter-robust convergence analysis of splitting methods that extends also to multiple network poroelastic theory (MPET) equations, see~\cite{HongEtAl2020fixed-stress,HongEtAl2020uzawa}.

However, no matter whether using classical block diagonal or block triangular preconditioners or splitting methods, the efficient solution of the subsystems for the individual physical fields remains an important issue. Domain decomposition and/or (algebraic) multigrid methods provide a feasible approach to solve the arising $H^1$ and $\Hdiv$ subproblems although their parameter-robust performance, e.g, for nearly incompressible materials and/or low hydraulic conductivities can be quite challenging.
In any case, this approach requires to combine at least two (nearly) optimal solvers for subproblems, such as elasticity and Darcy flow, 
in order to obtain an overall (nearly) optimal method.

A monolithic multigrid approach for a reduced-quadrature discretization has recently been developed in \cite{AdlerHeHuMaclachlanOhm23} where an additive Vanka relaxation method, based on overlapping vertex-patches, and an inexact Braess-Sarazin relaxation method is considered.
The methods rely on optimized relaxation factors depending on the Poisson ratio.
Their numerical tests confirm robustness of the method in the permeability coefficient and for moderately incompressible materials.
In comparison, our method is robust with respect to the Poisson ratio.

In this work we design a monolithic two-level Schwarz method based on overlapping vertex-patches for the quasi-static Biot problem in three-field formulation similar to \cite{AdlerHeHuMaclachlanOhm23}. However, instead of a reduced-quadrature discretization we apply the preconditioner to a family of strongly mass-conserving mixed finite element methods using an $\Hdiv$-conforming discontinuous Galerkin ansatz for the displacement field, as proposed in~\cite{HongKraus18,KanschatRiviere18}.
This allows us to prove convergence of the method by transforming the saddle-point problem into an equivalent singularly perturbed symmetric
positive definite problem, cf.~\cite{Schoeberl99dissertation,LeeWuXuZikatanov07}.
In particular, we obtain a robust method with standard choices of the relaxation parameters under the assumption of two-sided bounds on the permeability tensor.

The remainder of the article is organized as follows:
we introduce Biot's consolidation model and its conservative discretization in~\cref{sec:model}. The monolithic two-level algorithm is presented
in~\cref{sec:twolevelalgorithm} and analyzed in~\cref{sec:analysis}. In~\cref{sec:experiments} we provide numerical evidence of our theoretical findings.

\section{Model and discretization}
\label{sec:model}
Biot's consolidation model in three-field formulation couples the displacement field $\bu$ of the solid component, the seepage velocity $\bv$ of the fluid, and the fluid pressure $p$ in a fully saturated porous medium at constant temperature.
It is posed on a computational domain $\Omega\times(0,T)$ consisting of a spatial domain $\Omega \subset \R^d$, $d\in\{1,2,3\}$, where in this work we focus on the case $d=2$, and a time interval $(0,T)$, $T\in\R$. In strong form, Biot's quasi-static model of consolidation reads
\begin{gather*}
    \arraycolsep1pt
    \begin{matrix}
        -\div\left(2\mu\strain \bu
        +
        \lambda \div (\bu) \bm I \right)
        &&&+&
        \willis \nabla p
        &=& \bbf \quad \mbox{in } \Omega \times (0,T),
        \\
        &&\bv
        &+&\permeability \nabla p
        &=& \bm 0 \quad \mbox{in } \Omega \times (0,T),
        \\
        -\willis \div \partial_t \bu
        &-& \div \bv
        &-& \storage \partial_t p
        &=& g \quad \mbox{in } \Omega \times (0,T),
    \end{matrix}
\end{gather*}
with initial conditions $\bu(0)=\bu_0$, and $p(0)=p_0$ at time $t=0$.
The set of equations describes the momentum balance, Darcy's law, and the mass balance, respectively.
Here, the physical parameters $\lambda>0$ and $\mu>0$ denote the Lam\'e coefficients of linear elasticity, $\willis>0$ denotes the Biot-Willis constant, $\permeability$ is the symmetric and positive definite permeability coefficient of the porous medium, and $\storage\ge0$ is the specific storage coefficient.
Further, the strain tensor is defined by $\strain \bu := (\nabla \bu+\nabla \bu^T)/2$.
For an introduction to the system of equations, as well as a discussion of the physical parameters we refer the reader to \cite{DetournayCheng93,Coussy04}.

We apply a semi-discretization in time by the backward Euler method with time step size $\tau>0$, resulting in a
semi-discrete static problem of the form
\begin{gather}\label{eq:strongproblemeuler}
    \arraycolsep1pt
    \begin{matrix}
        - 2 \mu \div \strain \bu 
        -
        \lambda \nabla \div \bu
        &&&+&
        \willis \nabla p
        &=& \bbf \quad \mbox{in } \Omega,
        \\
        &&\bv
        &+&\permeability \nabla p
        &=& \bm 0 \quad \mbox{in } \Omega,
        \\
        -\frac\willis\tau \div \bu
        &-& \div \bv
        &-& \frac\storage\tau p
        &=& \bar g  \quad \mbox{in } \Omega,
    \end{matrix}
\end{gather}
in every time step where $\bar g := g -\frac\willis\tau \div \bu_{\rm old} - \frac\storage\tau p_{\rm old}$ with $\bu_{\rm old}$ and
$p_{\rm old}$ taken from the previous time moment and all other quantities referring to the new time moment. 
We point out that the specific time discretization is irrelevant for this article, since all implicit methods will result in the necessity
of solving systems of similar kind. Furthermore, we note that $\bu/\tau$ has the unit of speed and hence it is meaningful to add
the quantities $(\div \bu)/\tau$, $\div \bv$, and $p/\tau$ in the last equation.

Following \cite{HongKraus19}, by proper scaling and substitution of variables, we can transform~\eqref{eq:strongproblemeuler}
into a symmetric problem.
First, from~\eqref{eq:strongproblemeuler} we have
\begin{gather*}
    \arraycolsep1pt
    \begin{matrix}
        -\div \strain \bu
        -
        \frac\lambda{2\mu} \nabla \div \bu
        &&&+&
        \frac\willis{2\mu} \nabla p
        &=&\frac1{2\mu} \bbf &\quad \mbox{in } \Omega,
        \\
        &&\frac\willis{2\mu}\permeability^{-1} \bv
        &+& \frac\willis{2\mu} \nabla p
        &=& \bm 0 &\quad \mbox{in } \Omega ,
        \\
        - \div \bu
        &-& \frac\tau\willis \div \bv
        &-& \frac\storage\willis p
        &=& \frac\tau\willis \bar g &\quad \mbox{in } \Omega.
    \end{matrix}
\end{gather*}
Next, by substituting
\begin{gather}
    \tilde{\bv} = \frac\tau\willis \bv,
    \quad
    \tilde{p} = \frac\willis{2\mu} p,
    \quad
    \tilde{\bbf} = \frac1{2\mu} \bbf,
    \quad
    \tilde{g} = \frac\tau\willis \bar g - \div \bu_{\rm old} - \frac\storage\willis p_{\rm old},
    \label{eq:rescaledvariables}
    \\
    \scaledpermeability^{-1} = \frac{\willis^2}{2\mu\tau}\permeability^{-1},
    \quad
    \scaledlambda = \frac\lambda{2\mu},
    \quad
    \scaledstorage = \frac{2\mu\storage}{\willis^2},
    \label{eq:rescaledparameters}
\end{gather}
we obtain the rescaled system
\begin{gather}
    \label{eq:strongproblemscaled}
    \arraycolsep1pt
    \begin{matrix}
        -\div\strain \bu
        +
        \scaledlambda \nabla\left(\div \bu \right)
        &&&+&
        \nabla \tilde{p}
        &=& \tilde{\bbf} &\quad \mbox{in } \Omega,
        \\
        &&\scaledpermeability^{-1} \tilde{\bv}
        &+&\nabla \tilde{p}
        &=& \bm 0 &\quad \mbox{in } \Omega,
        \\
        -\div \bu
        &-& \div \tilde{\bv}
        &-& \scaledstorage \tilde{p}
        &=& \tilde{g} &\quad \mbox{in } \Omega.
    \end{matrix}
\end{gather}
Since this is the normalized system we build our analysis on, we will omit the tilde symbol from here on, but we will keep the hat to clarify the scaling of the parameters.

Throughout this paper, we assume that the scaled permeability tensor $\scaledpermeability$ is bounded such that
\begin{gather} \label{eq:normequivalence_k}
    \scaledRmin \ltwonorm{\bv}^2 \leq \ltwoscal(\scaledpermeability^{-1} \bv,\bv) \leq \scaledRmax \ltwonorm{\bv}^2
    \quad\forall \bv\in [L^2(\Omega)]^2,
\end{gather}
for two positive constants $\scaledRmin$ and $\scaledRmax$.
Here, $\ltwonorm{\cdot}=\norm{\cdot}_{L^2(\Omega)}$ and $\ltwoscal(\cdot,\cdot)=\scal(\cdot,\cdot)_{L^2(\Omega)}$ denote the $L^2$-norm
and $L^2$-inner product on $\Omega$, respectively.
Note that condition~\eqref{eq:normequivalence_k} follows, e.g., from the pointwise condition
\begin{gather*}
  \scaledRmin \leq \eigenvalue_{\min}(\scaledpermeability^{-1}) \leq \eigenvalue_{\max}(\scaledpermeability^{-1}) \leq \scaledRmax,
  \qquad \text{a. e. in } \Omega,
\end{gather*}
where $\eigenvalue_{\min}(\scaledpermeability^{-1})$ and $\eigenvalue_{\max}(\scaledpermeability^{-1})$ denote (pointwise) the smallest and largest eigenvalues of $\scaledpermeability^{-1}$, respectively.
As a matter of fact, the constants in the analysis below depend on $\scaledRmax$ and $\scaledRmin$, so that it neither intends to cover high contrast coefficients nor the case of strong anisotropy.
In the simplest case, we consider a permeability tensor of the form $\scaledpermeability=\scaledkappa \bm I$ with a scalar $\scaledkappa$, constant over $\Omega$ and constant over the time interval $(0,T)$.

This system is closed by proper boundary conditions. As usual, essential boundary conditions enter the definitions
of the function spaces, here $U\subset [H^1(\Omega)]^2$ and $V\subset \Hdiv(\Omega)$ for displacement and seepage velocity,
respectively.
We will choose a particular combination of boundary conditions in the presentation and analysis of the two-level algorithm
in Sections~\ref{sec:twolevelalgorithm} and~\ref{sec:analysis}, although one might also want to consider other combinations
for which certain arguments/basic estimates have to be adapted/exchanged accordingly.
The general assumptions that are required to conduct an analysis similar to the one presented in this paper are as follows.

\begin{assumption}
\label{ass:biot:boundary}
  The boundary conditions on the spaces $U$ and $V$ are such that:
  \begin{enumerate}
      \item The seepage velocity $\bv$ is uniquely determined.
      \item Korn's inequality holds for $\bu$ in the continuous setting as well as for $\bu_h$ in the discrete setting. For sufficient conditions
      in the continuous case we refer to~\cite{DesvillettesVillani02} and the literature cited there. For the discrete case, see~\cite{Brenner04Korn}.
      \item The pressure space $Q$ is chosen such that the operator $\div\colon V\to Q$ is surjective. The boundary conditions on $U$ are compatible such that also $\div\colon U\to Q$ is surjective.
  \end{enumerate}
\end{assumption}

To give one example, which applies to the setting of the analysis presented in Section~\ref{sec:analysis}, Korn's inequality
\begin{align} \label{eq:korn}
    \frac{1}{c_e}\ltwonorm{\nabla \bu}^2 \leq \ltwonorm{\strain \bu}^2 \leq \ltwonorm{\nabla \bu}^2
\end{align}
holds for all $\bu\in [H^1_0(\Omega)]^2$.

In line with this, on the whole boundary we assume a homogeneous Dirichlet (no slip) boundary condition for $\bu$ and a Neumann boundary condition for $p$, which translates to an essential condition for $\bv\cdot\bn$.
Accordingly, we choose
\begin{gather*}
    U=[H^1_0(\Omega)]^2, \qquad V=\Hdiv_0(\Omega), \qquad Q=L^2_0(\Omega),
\end{gather*}
where

\begin{align*}
    \Hdiv_0(\Omega) &= \set{\bv\in\Hdiv(\Omega) : \bv\cdot\bn=0 \text{ on } \partial\Omega},
\end{align*}
in the sense of traces of $\Hdiv$ and
\begin{align*}
    L^2_0(\Omega) &= \set{q\in L^2(\Omega):\int_{\Omega}q\:dx=0}.
\end{align*}

The weak formulation of \eqref{eq:strongproblemscaled} is then:
find $(\bu,\bv,p)\in U\times V\times Q$, such that
\begin{gather} \label{eq:modelproblem}
    \arraycolsep1pt
    \begin{array}{rcccccc@{\hspace{3em}}l}
        \eeform(\bu,\bphi)
        +
        \scaledlambda\ddform(\bu,\bphi)
        &&&-&
        \divform(p,\bphi)
        &=&\ltwoscal(\bbf,\bphi)
        &\forall \bphi \in U,
        \\
        &&\kform(\bv,\bpsi)
        &-&\divform(p,\bpsi)
        &=& \bm 0
        &\forall \bpsi \in V,
        \\
        -\divform(q,\bu)
        &-& \divform(q,\bv)
        &-& \scaledstorage \ltwoscal(p,q)
        &=&\ltwoscal(g,q)
        &\forall q\in Q,
    \end{array}
\end{gather}
with bilinear forms
\begin{align*}
    \eeform(\bu,\bphi) &= \ltwoscal(\strain \bu,\strain \bphi), &
    \ddform(\bu,\bphi) &= \ltwoscal(\div \bu,\div \bphi), \\
    \divform(p,\bpsi) &= \ltwoscal(p,\div \bpsi), &
    \kform(\bv,\bpsi) &= \ltwoscal(\scaledpermeability^{-1} \bv,\bpsi).
\end{align*}

\subsection{Discretization}
The model problem~\eqref{eq:modelproblem} is discretized by a family of mixed strongly mass-conserving methods
as proposed in~\cite{HongKraus18,KanschatRiviere18,HongKraus19}, where the displacement and seepage velocity
fields are approximated in suitable $\Hdiv$-conforming spaces $U_h$ and $V_h$ and the pressure space  $Q_h$
consists of piecewise polynomial functions discontinuous at element interfaces such that the conditions
\begin{gather}\label{eq:massconservationcondition}
    \div U_h = Q_h, \qquad \div V_h = Q_h
\end{gather}
are satisfied.
The combined finite element space for the mixed method is then defined by $X_h := U_h\times V_h\times Q_h$.
Since the space $U_h$ is not $H^1$-conforming, a discrete interior penalty discontinuous Galerkin bilinear form
$\eehform(\cdot,\cdot)$ is used to approximate $\eeform(\cdot,\cdot)$ as detailed below.

Let $\mesh$ be a family of shape regular triangulations of the computational domain $\Omega$ into mesh cells $\cell$
with diameter $h_\cell=\diam(\cell)$ where $h:=\max_{\cell\in\mesh}{h_\cell}$ denotes the mesh size.
Let $\setoffaces$ be the set of all interior faces (edges in two dimensions) of $\mesh$, and $\setofboundaries$ be the set
of all faces on the boundary $\partial\Omega$.
For every $\face\in\setoffaces$ there are two neighboring cells $\cell_+,\cell_-\in\mesh$ such that
$\face=\partial\cell_+\cap\partial\cell_-$. Let $\bn$ be the unit outward normal vector pointing from $\cell_+$ to $\cell_-$.
Then on every face $\face \in \setoffaces$ and for any $\bphi \in [L^2(\Omega)]^2$ and $\btau \in [L^2(\Omega)]^{2 \times 2}$,
we define jump $\jump{\, }$ and average
$\mean{\, }$ by
\begin{gather*}
    \jump{\bphi} = \bphi_+ - \bphi_-, \qquad
    \mean{\btau \bn} = \frac12\left(\btau_+ + \btau_- \right) \bn,
\end{gather*}
where $\bphi_{\pm}=\bphi|_{\cell_\pm}$ and $\btau_{\pm}=\btau|_{\cell_\pm}$.
Further on the broken Sobolev space
\begin{gather*}
    [H^1(\Omega,\mesh)]^2 = \left\{ \bphi\in [L^2(\Omega)]^2 \;\middle\vert\; \bphi|_\cell \in [H^1(\cell)]^2\right\}
\end{gather*}
we introduce the discrete norm
\begin{gather*}
    \dgnormone{\bphi}
    = \left(
    \sum_{\cell\in\mesh} \ltwonorm[\cell]{\nabla \bphi}^2
    + \sum_{\face\in\setoffaces} \frac1{h} \ltwonorm[\face]{\jump{\bphi}}^2
    + \sum_{\boundary\in\setofboundaries} \frac1{h} \ltwonorm[\boundary]{\bphi}^2
    \right)^{\frac12}.
\end{gather*}
On $U_h\times U_h$ we define the discrete bilinear form
\begin{align}
    \eehform(\bu_h,\bphi)
    =& \sum_{\cell\in\mesh} \ltwoscal[\cell](\strain{\bu_h},\strain{\bphi})
    + \sum_{\face\in\setoffaces} \frac\eta{h} \ltwoscal[\face](\jump{\bu_h},\jump{\bphi}) \nonumber\\
    &- \sum_{\face\in\setoffaces} \ltwoscal[\face](\mean{\strain{\bu_h} \bn},\jump{\bphi})
    - \sum_{\face\in\setoffaces} \ltwoscal[\face](\jump{\bu_h},\mean{\strain{\bphi} \bn}) \label{eq:eehform}\\
    &+ \sum_{\boundary\in\setofboundaries} \frac\eta{h} \ltwoscal[\boundary](\bu_h,\bphi)
    - \sum_{\boundary\in\setofboundaries} \ltwoscal[\boundary](\strain{\bu_h} \bn,\bphi)
    - \sum_{\boundary\in\setofboundaries} \ltwoscal[\boundary](\bu_h,\strain{\bphi} \bn). \nonumber
\end{align}
Here, the penalty parameter $\eta>0$ is chosen large enough to ensure coercivity of $\eehform(\cdot,\cdot)$, i.e.,
there is a positive constant $c$ such that
\begin{gather} \label{eq:coercivity_e}
    \eehform(\bphi,\bphi) \ge c \dgnormone{\bphi}^2 \quad\forall \bphi\in U_h.
\end{gather}
In addition we have continuity of $\eehform(\cdot,\cdot)$ in the norm $\dgnormone{\cdot}$, i.e.,
\begin{gather} \label{eq:continuity_e}
    \eehform(\bu_h,\bphi) \le c \dgnormone{\bu_h}  \dgnormone{\bphi} \quad\forall \bu_h,\bphi \in U_h.
\end{gather}
We note that since $U_h\subset \Hdiv(\Omega)$ 
the jumps in face terms in~\eqref{eq:eehform}
are zero in normal direction and hence can equivalently be defined using tangential components only. 

The mass-conserving mixed method based on the finite element space $X_h$ can then be represented as
\begin{gather} \label{eq:discreteproblem}
    \mixedform(\vvv{\bu_h}{\bv_h}{p_h},\vvv{\bphi}{\bpsi}{q})
    = \mixedrhs(\vvv{\bphi}{\bpsi}{q}),
\end{gather}
where the discrete bilinear form $\mixedform(\cdot,\cdot)\colon X_h\times X_h\to\R$ is defined by
\begin{equation}\label{eq:mixedform}
    \begin{aligned}
    \mixedform(\vvv{\bu_h}{\bv_h}{\ph},\vvv{\bphi}{\bpsi}{q})
    =& \eehform(\bu_h,\bphi) + \scaledlambda \ddform(\bu_h,\bphi) + \kform(\bv_h,\bpsi) \\
    & - \divform(p_h,\bphi+\bpsi)
    - \divform(q,\bu_h+\bv_h) - \scaledstorage \ltwoscal(p_h,q)
    \end{aligned}
\end{equation}
and the right hand side is given by
\begin{gather*}
    \mixedrhs(\vvv{\bphi}{\bpsi}{q}) = \ltwoscal(\bbf,\bphi) + \ltwoscal(g_h,q),
\end{gather*}
with $\gh$ chosen as the $L^2$-projection $\Ltwoprojection g$.
System \eqref{eq:discreteproblem} is consistent and uniformly well-posed, which has been proven in \cite{KanschatRiviere18,HongKraus18}
and follows essentially by the special choice of $\Hdiv$-conforming discretization spaces, coercivity \eqref{eq:coercivity_e},
continuity \eqref{eq:continuity_e} and the discrete inf-sup conditions,
see, e.g.~\cite{BoffiBrezziFortin13,HansboLarson02,SchoetzauSchwabToselli03,Brezzi74})
\begin{gather*} 
    \inf_{q\in Q_h} \sup_{\bphi\in U_h} \frac{\ltwoscal(\div \bphi,q)}{\dgnormone{\bphi} \ltwonorm{q}} \geq \constantinfsupu > 0
\end{gather*}
and
\begin{gather*} 
    \inf_{q\in Q_h} \sup_{\bpsi\in V_h} \frac{\ltwoscal(\div \bpsi,q)}{\Hdivnorm{\bpsi} \ltwonorm{q}} \geq \constantinfsupv > 0,
\end{gather*}
where
\begin{gather*}
    \Hdivnorm{\bpsi} = \left( \ltwonorm{\bpsi}^2 + \ltwonorm{\div\bpsi}^2 \right)^{\frac12}.
\end{gather*}

\section{Two-level algorithm}
\label{sec:twolevelalgorithm}

In this section, we define a monolithic two-level algorithm for the mixed problem~\eqref{eq:discreteproblem}.
The method differs from partitioned solvers which rely on and exploit the block structure of the system.

Let $\{\Omega_j\}_{j=1}^J$ be an overlapping covering of the domain $\Omega$ into subdomains $\Omega_j$ such that the following two standard assumptions are met, see~\cite{ToselliWidlund10,ArnoldFalkWinther97Hdiv,FengKarakashian01-1}.

\begin{assumption}[Finite covering]
\label{assumption:finite-covering}
The partition $\left\{\Omega_j\right\}_{j=1}^J$ can be colored using at most $N^c$ colors, in such a way that subregions with the same color are disjoint.
\end{assumption}

\begin{assumption}[Sufficient overlap]
\label{assumption:sufficient-overlap}
There is a constant $\delta>0$ that measures the size of overlaps between the subdomains $\Omega_j$, for $j=1,...,J$, such that
\begin{gather*}
    c_1 h\leq \delta \leq c_2 H
\end{gather*}
with constants $c_1>0$ and $c_2>0$.
\end{assumption}

In our experiments, we focus on the \emph{vertex patch smoother}, where the subdomains $\Omega_j$ can consist of all mesh cells sharing the vertex $j$, i.e., the vertex patch associated with vertex $j$.
In the special case of a uniform partition of the domain into rectangles this would be a union of $4$ mesh cells for interior vertices.

With the subdomains $\Omega_j$ we associate subspaces $X_j\subset X_h$.
To this end, we define
\begin{gather*}
    \begin{split}
        U_j := U_h\cap\Hdiv_0(\Omega_j),\quad
        V_j := V_h\cap\Hdiv_0(\Omega_j),\quad
        Q_j := Q_h\cap L^2_0(\Omega_j),
    \end{split}
\end{gather*}
continued by zero on $\Omega\setminus\Omega_j$, and set
\begin{gather*}
    X_j := U_j\times V_j\times Q_j.
\end{gather*}
Moreover, for $H>h$ let $\coarsemesh$ be a triangulation of the coarse domain $\Omega_0=\Omega$ such
that the associated subspace $X_0:=X_H \subset X_h$. Then, $X_0$ plays the role of a global coarse space and
we can decompose $X_h$ into an overlapping sum by
\begin{gather*}
    X_h=X_0 + \sum_{j=1}^J X_j.
\end{gather*}
For $j=1,\ldots,J$ we introduce local bilinear forms $\locmixedform(\cdot,\cdot)\colon X_j\times X_j\to \R$ as restrictions
of $\mixedform(\cdot,\cdot)$ to $X_j$ and define projections $\mixedP_j\colon X_h\to X_j$ by
\begin{gather}\label{eq:Pj}
    \locmixedform( \mixedP_j\vvv{\bu_h}{\bv_h}{\ph},\vvv{\bphi_j}{\bpsi_j}{q_j}) = \mixedform(\vvv{\bu_h}{\bv_h}{\ph},\vvv{\bphi_j}{\bpsi_j}{q_j})
    \qquad \forall \vvv{\bphi_j}{\bpsi_j}{q_j}\in X_j.
\end{gather}
Additionally, on the coarse space $X_0=X_H$, we introduce a coarse bilinear form $\locmixedform[0](\cdot,\cdot)\colon X_0\times X_0\to \R$ as $\locmixedform[0](\cdot,\cdot)=\coarsemixedform(\cdot,\cdot)$ and define $\mixedP_0\colon X_h\to X_0$ by
\begin{gather}\label{eq:P0}
    \locmixedform[0]( \mixedP_0\vvv{\bu_h}{\bv_h}{\ph},\vvv{\bphi_H}{\bpsi_H}{q_H}) = \mixedform(\vvv{\bu_h}{\bv_h}{\ph},\vvv{\bphi_H}{\bpsi_H}{q_H})
    \qquad \forall \vvv{\bphi_H}{\bpsi_H}{q_H}\in X_0.
\end{gather}
Note, that the bilinear form $\eeHform(\cdot,\cdot)$ is not inherited from $\eehform(\cdot,\cdot)$ although $X_H\subset X_h$, because it differs in the face and boundary terms.
Thus, $\mixedP_0$ is not a projection, since consequently $\locmixedform[0](\cdot,\cdot)$ differs from $\mixedform(\cdot,\cdot)$.

The bilinear form $\mixedform(\cdot,\cdot)\colon X_h\times X_h\to\R$ given in~\eqref{eq:mixedform} defines
an operator $\systemoperator := \systemoperator_h: X_h \rightarrow X'_h$ via the relation
\begin{gather*}
\langle \systemoperator \bx_h, \by_h \rangle_{X'_h \times X_h} = \mixedform(\bx_h, \by_h) \qquad \forall \bx_h, \by_h \in X_h,
\end{gather*}
where $X'_h$ denotes the dual space of $X_h$ and $\langle \cdot, \cdot \rangle_{X'_h \times X_h}$ the corresponding duality pairing.
With this notation we define the two-level additive Schwarz preconditioner $\Pad: X'_h \rightarrow X_h$ by
\begin{gather}\label{eq:additiveschwarz}
    \Pad = \relaxation \sum_{j=0}^J \mixedP_j \systemoperator^{-1},
\end{gather}
where $\relaxation$ is a relaxation factor depending on the overlap that is also used for tuning the method.
Further, the multiplicative two-level Schwarz preconditioner $\Pmu$ is defined by
\begin{align}
    \label{eq:multiplicativeschwarz}
    \Pmu &= \left(I - \Emu\right) \systemoperator^{-1},
    \quad
    \Emu = \left(I - \mixedP_J\right)\cdots\left(I - \mixedP_1\right)\left(I - \mixedP_0\right).
\end{align}
Note, that in practice only local problems are solved and the global inverse $\systemoperator^{-1}$ is actually never computed in the application of the preconditioners.

\section{Two-level convergence analysis}
\label{sec:analysis}

We begin this section by stating the main results:

\begin{theorem}
\label{theorem:additive}
Assume that $H \le c h$ for some positive constant $c$ and that the bounds~\eqref{eq:normequivalence_k} on $\scaledpermeability$ hold.
Then, the additive Schwarz method~\eqref{eq:additiveschwarz} defines a uniform preconditioner for the system~\eqref{eq:discreteproblem}.
Consequently, the GMRES method converges uniformly, i.e.,
it achieves a prescribed accuracy within a number of iterations
that is bounded independently of the mesh size $h$ and any other model parameters than $\scaledpermeability$, when applied to solve a linear system with $\Pad \systemoperator$.
\end{theorem}

\begin{theorem}
\label{theorem:multiplicative}
   Under the assumptions of \cref{theorem:additive} the stationary iterative method based on the multiplicative Schwarz  
   preconditioner~\eqref{eq:multiplicativeschwarz} converges uniformly, i.e., with a residual contraction factor strictly less than $1$, independent of the mesh size $h$ and any other model parameters than $\scaledpermeability$.
\end{theorem}

\begin{remark}
The dependence of the presented two-level convergence estimates on the parameter $\scaledpermeability$ is caused by the same dependence of the constant in the stable decomposition of the divergence-free subspace $V_h^0$ of $V_h$ under the norm induced by the bilinear form $\kform(\cdot,\cdot)$, see~\eqref{eq:ltwo_decomp_hdivconf} in the proof of Lemma~\ref{lemma:decomposition_kernel}.
We will use a very crude upper bound at this point, which is given by the global contrast $\scaledRmax/\scaledRmin$ for reasons of simplicity because this aspect is not the focus of the presented analysis.
Improved bounds for Darcy have been developed based on weighted Poincar\'e inequalities under the assumption of certain quasi-monotonicity properties of the coefficient function $\scaledpermeability^{-1}$ in  $\kform(\cdot,\cdot)$, see~\cite{Pechstein2011weighted,Scheichl2012multilevel,Pechstein2013weighted}. 
\end{remark}

Due to the lack of an abstract convergence theory for subspace correction methods for indefinite problems, 
we follow~\cite{Schoeberl99dissertation} and transform the saddle point problem into an equivalent, singularly perturbed 
symmetric positive definite (SPD) system, which is possible for any fixed storage coefficient $\scaledstorage > 0$, as 
detailed in Section~\ref{sec:singularperturbedproblem}.

The analysis of domain decomposition methods for this singularly perturbed problem requires a stable decomposition derived in subsections~\ref{sec:space-decomposition} and~\ref{sec:stabledecomposition}.
Thus, the convergence with respect to the singularly perturbed problem follows by the theory for symmetric positive definite problems, which is proved in \cref{section:SPP-convergence}.
Finally, the proofs of \cref{theorem:additive,theorem:multiplicative} can be found in subsection~\ref{sec:proof}. They rely on the fact that the estimates for the singularly perturbed problem are uniform for $\scaledstorage \to 0$.

\subsection{An equivalent, singularly perturbed problem}\label{sec:singularperturbedproblem}
As shown in \cite{HongKraus19}, strong mass conservation can be recovered from system~\eqref{eq:discreteproblem}
by using $\Hdiv$-conforming discretizations for the displacement field and seepage velocity in combination with piecewise
polynomial approximations of the pressure if the polynomial degrees match. To be more precise, the discrete mass balance
equation
\begin{gather*}
    -\div \bu_h - \div \bv_h -\scaledstorage \ph = \gh
\end{gather*}
is fulfilled point-wise if the condition~\eqref{eq:massconservationcondition} is met.
This allows us to substitute
\begin{gather*}
    \ph = - \scaledstorage^{-1} \left(\div \bu_h + \div \bv_h + \gh\right)
\end{gather*}
and obtain the equivalent, singularly perturbed system
\begin{equation*}
\begin{aligned}
    \eehform(\bu_h,\bphi) + \scaledlambda \ddform(\bu_h,\bphi) + \scaledstorage^{-1}\ddform(\bu_h+\bv_h,\bphi) &= \ltwoscal(\bbf,\bphi) - \scaledstorage^{-1} \ltwoscal(\gh,\div\bphi)\\
    \kform(\bv_h,\bpsi) +  \scaledstorage^{-1} \ddform(\bu_h+\bv_h,\bpsi) &= - \scaledstorage^{-1} \ltwoscal(\gh,\div \bpsi)
\end{aligned}
\end{equation*}
for all $(\bphi,\bpsi)\in \Wh = U_h\times V_h$.
We rewrite this system as
\begin{gather} \label{eq:singularproblem}
    \singularform(\vv{\bu_h}{\bv_h},\vv{\bphi}{\bpsi})
    = \singularrhs(\vv{\bphi}{\bpsi})
\end{gather}
with
\begin{gather}\label{eq:SPP-discretebilinearform}
    \singularform(\vv{\bu_h}{\bv_h},\vv{\bphi}{\bpsi})
    = \Ehform(\vv{\bu_h}{\bv_h},\vv{\bphi}{\bpsi})
    + \Dform(\vv{\bu_h}{\bv_h},\vv{\bphi}{\bpsi}),
\end{gather}
where the bilinear forms $\Ehform(\cdot,\cdot)$ and $\Dform(\cdot,\cdot)$
are defined on $\Wh\times\Wh$ by
\begin{align*}
    \Ehform(\vv{\bu_h}{\bv_h},\vv{\bphi}{\bpsi})
    &= \eehform(\bu_h,\bphi) + \kform(\bv_h,\bpsi), 
    \\
    \Dform(\vv{\bu_h}{\bv_h},\vv{\bphi}{\bpsi})
    &= \scaledlambda \ddform(\bu_h,\bphi) + \scaledstorage^{-1} \ddform(\bu_h+\bv_h,\bphi+\bpsi) 
\end{align*}
and the right hand side is given by
\begin{gather*}
    \singularrhs(\vv{\bphi}{\bpsi}) = \ltwoscal(\bbf,\bphi) - \scaledstorage^{-1} \ltwoscal(\gh,{\div(\bphi + \bpsi)}).
\end{gather*}

\begin{remark}
  \label{lemma:biot:posdef}
  The form $\singularform(\cdot,\cdot)$ is symmetric and coercive with respect to the weighted norm
  \begin{gather*}
      \Whnorm{\bw_h} = \left(\dgnormone{\bu_h}^2 + \scaledlambda \ltwonorm{\div\bu_h}^2 + \ltwonorm{\scaledpermeability^{-\nfrac12}\bv_h}^2 + \scaledstorage^{-1} \ltwonorm{\div (\bu_h+\bv_h)}^2 \right)^{\frac12}
  \end{gather*}
  for all $\bw_h=(\bu_h,\bv_h)\in\Wh$, i.e.,
  \begin{gather*}
    \singularform(\vv{\bu_h}{\bv_h},\vv{\bu_h}{\bv_h})
    \geq c\Whnorm{\bw_h}^2.
  \end{gather*}
\end{remark}

Similar to \cref{sec:twolevelalgorithm} we choose local subspaces $\Wj \subset\Wh$ associated with the subdomains $\Omega_j$,
such that
\begin{gather*}
    \Wj = U_j\times V_j
\end{gather*}
as well as a global coarse space $W_0=U_H\times V_H$. Hence, every $\bw\in \Wh$ admits a decomposition of the form
\begin{gather}\label{eq:decomposition_w}
    \bw = \sum_{j=0}^J \bw_j,
    \qquad w_j\in \Wj.
\end{gather}

For $j=1,\ldots,J$ we introduce the local bilinear forms $\locsingularform(\cdot,\cdot)\colon\Wj\times\Wj\to\R$ as restrictions of
$\singularform(\cdot,\cdot)$ to $\Wj$ and a coarse bilinear form $\locsingularform[0](\cdot,\cdot)\colon W_0\times W_0\to\R$
as $\locsingularform[0](\cdot,\cdot)=\coarsesingularform(\cdot,\cdot)$. Analogously to~\eqref{eq:Pj} and~\eqref{eq:P0} we define
the projections $\singularP_j\colon W_h\to W_j$ by
\begin{gather*}
    \locsingularform( \singularP_j\vv{\bu_h}{\bv_h},\vv{\bphi_j}{\bpsi_j}) = \singularform(\vv{\bu_h}{\bv_h},\vv{\bphi_j}{\bpsi_j})
    \qquad \forall \vv{\bphi_j}{\bpsi_j}\in W_j
\end{gather*}
and the projection-like operator $\singularP_0\colon W_h\to W_0$ by
\begin{gather*}
    \locsingularform[0]( \singularP_0\vv{\bu_h}{\bv_h},\vv{\bphi_H}{\bpsi_H}) = \coarsesingularform(\vv{\bu_h}{\bv_h},\vv{\bphi_H}{\bpsi_H})
    \qquad \forall \vv{\bphi_H}{\bpsi_H}\in W_0.
\end{gather*}

Furthermore, the bilinear form $\singularform(\cdot,\cdot)\colon \Wh\times \Wh\to\R$ given in~\eqref{eq:SPP-discretebilinearform} defines
an operator $\singularoperator:=\singularoperator_h: \Wh \rightarrow \Wh[']$ via the relation
\begin{gather*}
    \langle \singularoperator \bx_h, \by_h \rangle_{\Wh['] \times \Wh} = \singularform(\bx_h, \by_h), \quad \forall \bx_h, \by_h \in \Wh.
\end{gather*}

Thus, for the singularly perturbed problem the two-level additive Schwarz preconditioner $\singularPad: \Wh['] \rightarrow \Wh$ is defined by
\begin{gather} \label{eq:SPP-additiveschwarz}
    \singularPad = \relaxation \sum_{j=0}^J \singularP_j \singularoperator^{-1},
\end{gather}
and the multiplicative two-level Schwarz preconditioner $\singularPmu$ is defined by
\begin{align}
    \label{eq:SPP-multiplicativeschwarz}
    \singularPmu &= \left(I - \singularEmu\right) \singularoperator^{-1},
    \quad
    \singularEmu = \left(I-\singularP_J\right)\cdots\left(I-\singularP_1\right)\left(I-\singularP_0\right).
\end{align}

\subsection{Decomposition of the spaces according to \texorpdfstring{$\kerDform$}{ker(D)}}
\label{sec:space-decomposition}
Our goal is to prove that for any $\bw \in W_h$ any decomposition of the form~\eqref{eq:decomposition_w}  
is stable under the global and local bilinear forms $\singularform(\cdot,\cdot) \colon W_h \times W_h\to\R$
and $\locsingularform(\cdot,\cdot)\colon\Wj\times\Wj\to\R$, respectively.
As a preparatory, step we decompose the space $\Wh$ into the orthogonal sum
\begin{align*}
    \Wh = \kerDform \oplus \kerDform^\perp
\end{align*}
of the kernel of the summed divergence operator
\begin{gather*}
    W_h^0 := \kerDform = \left\{\vv{\bu}{\bv}\in W_h\middle|\; \Dform(\vv{\bu}{\bv},\vv{\bphi}{\bpsi}) = 0 \quad\forall \vv{\bphi}{\bpsi} \in W_h\right\}
\end{gather*}
and its $\Ehform(\cdot,\cdot)$-orthogonal complement
\begin{gather*}
    \kerDform^\perp = \left\{\vv{\bu}{\bv}\in W_h \middle|\; \Ehform(\vv{\bu}{\bv},\vv{\bphi}{\bpsi})=0\quad \forall\vv{\bphi}{\bpsi}\in\kerDform\right\},
\end{gather*}
cf.~\cite{ArnoldFalkWinther97Hdiv,Hiptmair97}.
To this end, we introduce the discrete Helmholtz-decompo\-sitions
\begin{gather*}
    U_h = U_h^0 \oplus U_h^\perp,
    \qquad
    V_h = V_h^0 \oplus V_h^\perp, 
\end{gather*}
of $U_h$ and $V_h$ into their divergence free subspaces $U_h^0$ and $V_h^0$
and their $\eehform(\cdot,\cdot)$-ortho\-gonal and $\kform(\cdot,\cdot)$-orthogonal
complements, respectively,
\begin{align*}
    U_h^\perp &= \left\{ \bu\in U_h \middle|\; \eehform(\bu,\bphi)=0\quad\forall \bphi\in U_h^0\right\},\\
    V_h^\perp &= \left\{ \bv\in V_h \middle|\; \kform(\bv,\bpsi)=0\quad\forall\bpsi\in V_h^0\right\},
\end{align*}
cf.~\cite{KanschatMao15,ArnoldFalkWinther10}.
Accordingly, we denote the Helmholtz-decompositions of the local and coarse spaces $U_j$ and $V_j$ into their divergence-free subspaces and their $\eehform(\cdot,\cdot)$-orthogo\-nal and $\kform(\cdot,\cdot)$-orthogonal complements by
\begin{gather*}
    U_j = U_j^0 \oplus U_j^\perp,
    \qquad
    V_j = V_j^0 \oplus V_j^\perp.
\end{gather*}

The following lemma characterizes $W_h^0 = \kerDform \subset W_h$.
\begin{lemma}
\label{lemma:decompproperties}
For $\scaledlambda>0$ there holds
\begin{gather*}
    \Wh[0] := \kerDform = U_h^0 \times V_h^0,
    \qquad
    \Wh[\perp] := \kerDform^\perp = U_h^\perp \times V_h^\perp.
\end{gather*}
\end{lemma}

\begin{proof}
Setting $\bphi=\bu$ and $\bpsi=\bv$ in
  \begin{gather*}
      \Dform(\vv{\bu}{\bv},\vv{\bphi}{\bpsi}) = \scaledlambda\ddform(\bu,\bphi)
    + \scaledstorage^{-1}\ddform(\bu+\bv,\bphi+\bpsi)
  \end{gather*}
shows that a necessary condition for $(\bu,\bv)^T \in \kerDform$ is $\div \bu = 0$, i.e., $\bu\in U_h^0$.
Further, for any $\bu\in U_h^0$, choosing $\bphi=\bu$ and $\bpsi=-\bv$ shows that in this case $\bv \in V_h^0$
is necessary and sufficient to guarantee that $(\bu,\bv)^T \in \Wh[0]$.
\end{proof}

Note, that for all $\bw_h^0 = (\bu_h^0,\bv_h^0) \in \Wh[0]$ and $\bpi_h = (\bphi_h,\bpsi_h)\in\Wh$, there holds
\begin{gather}
  \label{eq:globaldecompkernel}
  \singularform(\bw_h^0, \bpi_h) =
  \eehform(\bu_h^0,\bphi_h) + \kform(\bv_h^0,\bpsi_h).
\end{gather}

\subsection{Stable decompositions}
\label{sec:stabledecomposition}
The main result of this section is \cref{lemma:StableDecomp}. Before we can prove it, we begin by recalling known facts from the literature
and showing some auxiliary results.
We start with the existence and some standard properties of a partition of unity for the overlapping covering $\{\Omega_j\}_{j=1}^J$ in \cref{prop:partition-of-unity}, see~\cite[Lemma 3.4]{ToselliWidlund10}.

\begin{proposition}
\label{prop:partition-of-unity}
Given \cref{assumption:finite-covering,assumption:sufficient-overlap}, there is a piecewise bilinear
partition of unity $\left\{\theta_j\right\}_{j=1}^J$ relative to the overlapping partition $\left\{\Omega_j\right\}_{j=1}^J$, i.e.,
$\operatorname{supp}(\theta_j) \subset \overline{\Omega_j}$, such that for all $x\in\overline{\Omega}$ there holds
\begin{equation} \label{eq:partitionofunity}
    \sum_{j=1}^J \theta_j(x) = 1, \quad \mbox{with } 
    0 \le \theta_j(x) \le 1,  \quad \mbox{for all } j=1,2,\ldots,J.
\end{equation}
Moreover,
\begin{equation*} 
    \linftynorm[(\Omega)]{\nabla \theta_j} \le \frac{c}{\delta},
\end{equation*}
where $\delta$ is the constant from \cref{assumption:sufficient-overlap}.
\end{proposition}

In order to represent divergence free functions in the displacement
space $U_h^0$, we introduce the space
\begin{gather*}
    S(\Omega,\mesh) = \left\{ s\in H^1_0(\Omega)\;\middle\vert\; s_{|\cell} \in H^2(\cell)\right\},
\end{gather*}
which is equipped with the discrete norm
\begin{gather}
  \label{eq:norm-2h}
    \dgnormtwo{s} = \left(
    \sum_{\cell\in\mesh}
    \ltwonorm[\cell]{\nabla^2 s}^2
    + \sum_{\face\in\setoffaces} \frac1{h} \ltwonorm[\face]{\jump{\partial_n s}}^2
    + \sum_{\boundary\in\setofboundaries} \frac1{h} \ltwonorm[\boundary]{\partial_n s}^2
    \right)^{\frac12}.
\end{gather}
Following~\cite{KanschatSharma14}, the bilinear form
$\eehform(\cdot,\cdot)$ on the divergence free subspace is
algebraically equivalent to the bilinear form of a $C^0$-interior
penalty formulation of a corresponding biharmonic problem by assigning
velocities $\bu_h = \curl s$, where $\bu_h$ belongs to the discrete stream function space $S(\Omega,\mesh)$.
Combining arguments from~\cite{KanschatSharma14} and~\cite{BrennerWang05}, it is easy to see that in two space dimensions there holds the elementary identity
\begin{gather}
  \label{eq:equivalence_curlH1_H2}
  \dgnormone{\curl s}
  = \dgnormtwo{s}, \quad \forall s\in S(\Omega,\mesh),
\end{gather}
where $\curl = (\partial_2,-\partial_1)$.

Next, we recall a result showing that stream functions in $S(\Omega,\mesh)$ have a stable decomposition with respect to the norm~\eqref{eq:norm-2h}.

\begin{lemma} \label{lemma:decomposition_H2}
Every $s\in S(\Omega,\mesh)$ admits a decomposition $s=\sum_{j=0}^J s_j$ with $s_j\in S(\Omega_j,\mesh)$, $s_0\in S(\Omega_0,\coarsemesh)$, such that
\begin{align}
\label{eq:decomposition_H2}
    \dgnormtwo[H]{s_0}^2 + \sum_{j=1}^J \dgnormtwo{s_j}^2 \leq c \left(1+\frac{H^4}{\delta^4}\right) \dgnormtwo{s}^2
\end{align}
for some constant $c>0$.
\end{lemma}

\begin{proof}
  Compare the proof of Lemma 4.2 in~\cite{BrennerWang05}.
\end{proof}

The above results suffice to prove the stability of the decomposition of the divergence free subspace of $U_h$ with respect to the elasticity energy norm.
\begin{lemma}
\label{lemma:decomposition_e_null}
Every function $\bu^0\in U_h^0$ admits a decomposition of the form $\bu^0=\sum_{j=0}^J \bu^0_j$ with $\bu^0_j\in U_j^0$, which satisfies the bound
\begin{gather*}
    \eeHform(\bu^0_0,\bu^0_0) + \sum_{j=1}^J \eehform(\bu^0_j,\bu^0_j)
    \leq c \left(1 + \frac{H^4}{\delta^4}\right) \eehform(\bu^0,\bu^0)
\end{gather*}
for some constant $c>0$.
\end{lemma}

\begin{proof}
For every $\bu^0\in U_h^0$ there exists a unique stream function $s\in S(\Omega,\mesh)$, such that
\begin{align*}
    \bu^0 = \curl s.
\end{align*}
Now we decompose $s = \sum_{j=0}^J s_j$ and choose
\begin{gather*}
    \bu^0_j = \curl s_j,
\end{gather*}
such that $\bu^0_j\in U_j^0$ and by linearity of the curl
\begin{gather*}
    \bu^0 = \sum_{j=0}^J \bu^0_j.
\end{gather*}
Continuity of $\eehform(\cdot,\cdot)$, the stability of the decomposition in~\eqref{eq:decomposition_H2}, the identity~\eqref{eq:equivalence_curlH1_H2}, 
and Korn's inequality~\eqref{eq:coercivity_e} yield
\begin{multline*}
    \eeHform(\bu^0_0,\bu^0_0) + \sum_{j=1}^J \eehform(\bu^0_j,\bu^0_j)
    \leq
    c \left(\dgnormtwo[H]{s_0}^2 + \sum_{j=1}^J \dgnormtwo{s_j}^2 \right)\\
    \leq c \left(1+\frac{H^4}{\delta^4}\right)\dgnormtwo{s}^2 
    = c \left(1+\frac{H^4}{\delta^4}\right) \dgnormone{\curl s}^2\\
    \leq c \left(1+\frac{H^4}{\delta^4}\right) \eehform(\bu^0,\bu^0).
\end{multline*}
\end{proof}

In order to show that also $U_h^\perp$ has a stable decomposition with respect to the norm induced by the elasticity bilinear form, we will adapt the construction from~\cite[Chapter10]{ToselliWidlund10}.

For this purpose, we recall the definition of the spaces
\begin{align*}
\Hdiv_0(\Omega) &:= \{ \bv \in \Hdiv(\Omega): \bv \cdot \bn =0 \mbox{ on } \partial \Omega \}, \\
\Hcurl_0(\Omega) &:= \{ v \in \Hcurl(\Omega): v \cdot \btau =0 \mbox{ on } \partial \Omega \},
\end{align*}
and decompose them into
\begin{align}
\Hdiv_0(\Omega) &= H_0^{\div,0}(\Omega) \oplus H_0^{\div,\perp}(\Omega), \label{decomp:H_0_div} \\
\Hcurl_0(\Omega) &= H_0^{\curl,0}(\Omega) \oplus H_0^{\curl,\perp}(\Omega) \label{decomp:H_0_curl},
\end{align}
where $H_0^{\div,\perp}(\Omega)$ and $H_0^{\curl,\perp}(\Omega)$ denote the orthogonal complements of the
divergence-free and curl-free subspaces
\begin{align*}
H_0^{\div,0}(\Omega) &:= \{ \bv \in \Hdiv_0(\Omega): \div \bv =0 \},  \\
H_0^{\curl,0}(\Omega) &:= \{ v \in \Hcurl_0(\Omega): \curl v ={\bm 0} \},
\end{align*}
of $\Hdiv_0(\Omega)$ and $\Hcurl_0(\Omega)$ in the $(\cdot,\cdot)_{\Hdiv(\Omega)}$ and $(\cdot,\cdot)_{\Hcurl(\Omega)}$ inner products, respectively.
Note that the decompositions~\eqref{decomp:H_0_div} and~\eqref{decomp:H_0_curl} are both $L^2$-orthogonal as well.

Now, we introduce a new space $U^+$ containing semicontinuous functions,
which can be characterized as the image of $U_h$ under an orthogonal projection
$\Theta^\perp \colon \Hdiv_0(\Omega)\to H_0^{\div,\perp}(\Omega)$, i.e., $U^+:=\Theta^\perp(U_h)$, which is defined by
\begin{align*}
    \Theta^\perp \bu := \bu - \curl w,
\end{align*}
where $w\in H_0^{\curl,\perp}(\Omega)$ satisfies
\begin{align*}
    \ltwoscal(\curl w, \curl v) = \ltwoscal(\bu, \curl v), \quad \forall v\in H_0^{\curl,\perp}(\Omega).
\end{align*}
Further, we define the projection $P^h\colon \Hdiv_0(\Omega)\to U^+$ by
\begin{align*}
    \ltwoscal(\div{\left(P^h \bu-\bu\right)}, \div\bphi) = 0, \quad \forall \bphi \in U^+.
\end{align*}
Then, for $P^h$ there holds
\begin{gather*}
    \Hdivnorm{P^h \bu^\perp} \le c \Hdivnorm{\bu^\perp}, \quad \forall \bu^\perp \in U_h^\perp
\end{gather*}
with a constant $c$ depending on $\Omega$.
According to~\cite[Proposition 4.6]{CockburnKanschatSchoetzau05} 
there is another constant $c$, independent of the mesh size, but depending on shape-regularity of the mesh and the polynomial degree, such that
\begin{gather}
  \label{eq:PH-bounded-discrete}
  \dgnormone{P^h \bu^\perp} \le c \dgnormone{\bu^\perp}.
\end{gather}
Moreover, from~\cite[Lemma 10.12]{ToselliWidlund10} for convex $\Omega$, we have the approximation property
\begin{gather}
    \label{eq:Ph-L2-error}
    \ltwonorm{u^\perp-P^h \bu^\perp} \le c \, h \ltwonorm{\div \bu^\perp}
    \qquad\forall \bu^\perp\in U_h^\perp.
\end{gather}
For the $L^2$-projection $\Ltwoprojection[H]\colon L^2(\Omega)^d\to U_0$ onto the coarse space $U_0$ there holds
\begin{align} \label{eq:L2error_L2projection}
    \ltwonorm{\bu - \Ltwoprojection[H] \bu}
    \leq c \, H \dgnormone{\bu}.
\end{align}
Additionally, we have $H^1$-stability of $\Ltwoprojection[H]$, that is, 
\begin{align}
    \label{eq:H1stabilityL2projection}
    \dgnormone{\Ltwoprojection[H]\bu} \leq c \dgnormone{\bu}.
\end{align}
We are ready to prove the following lemma.

\begin{lemma} \label{lemma:decomposition_e_perp}
Every function $\bu^\perp\in U_h^\perp$ admits a decomposition of the form $\bu^\perp=\sum_{j=0}^J \bu^1_j$ with $\bu^1_j\in U_j$,
which satisfies the bound
\begin{gather}\label{eq:eh-decomp-global}
    \eeHform(\bu^1_0,\bu^1_0) + \sum_{j=1}^J \eehform(\bu^1_j,\bu^1_j)
    \leq c \left(1 + \frac{H^2}{\delta^2}\right) \eehform(\bu^\perp,\bu^\perp)
\end{gather}
for some constant $c>0$.
\end{lemma}

\begin{proof}
We choose
\begin{align*}
    \bu^1_0 &= \Ltwoprojection[H] P^h \bu^\perp, \\
    \bu^1_j &= \interpolant\left(\theta_j \left(\bu^\perp-\bu^1_0\right)\right),
\end{align*}
where $\interpolant$ is the canonical interpolation into $U_h$ and $\{\theta_j\}_{j=1}^J$ the piecewise linear partition of unity defined
in~\eqref{eq:partitionofunity}.
We point at the fact that $\bu^1_j\in U_j$ is not necessarily a function in $U_j^\perp$.

Let $\tilde{\bu} = \bu^\perp-\bu^1_0$.
Noting that on each cell $\cell\in\mesh$, the function $\theta_j \tilde{\bu}$ is polynomial, such that there holds
\begin{gather}
    \label{eq:H1stabilityInterpolation}
    \dgnormone{\interpolant\left(\theta_j \tilde \bu\right)} \leq c \dgnormone{\theta_j \tilde \bu}.
\end{gather}
By continuity of the bilinear form $\eeHform(\cdot,\cdot)$, the bounds~\eqref{eq:H1stabilityL2projection} and~\eqref{eq:PH-bounded-discrete},
and coercivity of $\eehform(\cdot,\cdot)$, we get
\begin{align}
  \label{eq:eh-decomp-coarse}
  \eeHform(\bu^1_0,\bu^1_0)
  &\leq c\dgnormone{\Ltwoprojection[H]P^h \bu^\perp}^2 \nonumber \\
  &\leq c\dgnormone{P^h \bu^\perp}^2
  \leq c\dgnormone{\bu^\perp}^2
  \le c\, \eehform(\bu^\perp,\bu^\perp)
\end{align}
and
\begin{gather}
    \label{eq:H1error_utilde}
    \dgnormone{\tilde{\bu}}^2
    \le c \left(\dgnormone{\bu^\perp}^2 + \dgnormone{\Ltwoprojection[H] P^h \bu^\perp}^2 \right)
    \le c\, \eehform(\bu^\perp,\bu^\perp).
\end{gather}
By \eqref{eq:Ph-L2-error}, \eqref{eq:L2error_L2projection}, and using $h\leq H$, we obtain
\begin{equation} \label{eq:L2error_utilde}
    \begin{aligned}
    \ltwonorm{\tilde{\bu}}
    &\leq \ltwonorm{\bu^\perp - P^h \bu^\perp} + \ltwonorm{P^h \bu^\perp - \Ltwoprojection[H]P^h \bu^\perp}
    \\
    &\leq c \, h \ltwonorm{\div \bu^\perp} + c \, H \dgnormone{\bu^\perp}
    \\
    &\leq c \, H \dgnormone{\bu^\perp}.
    \end{aligned}
\end{equation}
Using the continuity of $\eehform(\cdot,\cdot)$, \eqref{eq:H1stabilityInterpolation}, the properties of $\theta_j$,
\eqref{eq:L2error_utilde}, \eqref{eq:H1error_utilde}, and the coercivity of $\eehform(\cdot,\cdot)$, as well as the 
finite covering~\cref{assumption:finite-covering}, we get
\begin{equation}
\label{eq:eh-decomp-local}
\begin{split}
  \sum_{j=1}^J \eehform(\bu^1_j,\bu^1_j)
  &\le c \sum_{j=1}^J \dgnormone{\interpolant\left(\theta_j \tilde{\bu} \right)}^2
  \le c \sum_{j=1}^J \dgnormone{\theta_j \tilde{\bu} }^2
  \\
  &\le c \sum_{j=1}^J \Biggl( \sum_{\cell\in\mesh} \left( \linftynorm[(\cell)]{\nabla\theta_j}^2 \ltwonorm[\cell]{\tilde{\bu}}^2 + \ltwonorm[\cell]{\theta_j \nabla \tilde{\bu}}^2 \right)
  \\
  &\qquad\qquad
         + \sum_{\face\in\setoffaces} \frac1{h} \ltwonorm[\face]{\theta_j \jump{\tilde{\bu} }}^2
         + \sum_{\boundary\in\setofboundaries} \frac1{h} \ltwonorm[\boundary]{\theta_j \tilde{\bu} }^2 \Biggr)
  \\
  &\le c \frac{1}{\delta^2} \ltwonorm{\tilde{\bu}}^2
  + c \dgnormone{\tilde{\bu}}^2
  \\
  &\le c \left(1+\frac{H^2}{\delta^2} \right) \eehform(\bu^\perp,\bu^\perp).
\end{split}
\end{equation}
Finally, combining~\eqref{eq:eh-decomp-coarse} and \eqref{eq:eh-decomp-local} results in the desired estimate~\eqref{eq:eh-decomp-global}.
\end{proof}

The last preparatory step is to show the stability of the decomposition of
$\Wh[0]=\bigcup_{j=0}^J \Wj[0]$ with respect to $\singularform(\cdot,\cdot)$.

\begin{lemma} \label{lemma:decomposition_kernel}
Every $\bw^0=(\bu^0,\bv^0)\in \Wh[0]$ admits a decomposition of the form
$\bw^0 = \sum_{j=0}^J \bw^0_j$ with $\bw^0_j \in \Wj[0]$, which satisfies the bound
\begin{gather*}
    \sum_{j=0}^J \locsingularform(\bw^0_j,\bw^0_j)
    \leq \constK \left(1+\frac{H^4}{\delta^4}\right) \singularform(\bw^0,\bw^0),
\end{gather*}
where $\constK = c \nfrac{\scaledRmax}{\scaledRmin}$ for some constant $c>0$ independent of the model parameters and of $J$.
\end{lemma}

\begin{proof}
By \cref{lemma:decomposition_e_null}, every $\bu^0\in U_h^0$ has a decomposition $\bu^0=\sum_{J=0}^J \bu^0_j$ with $\bu^0_j\in U_j^0$, such that
\begin{align*}
    \eeHform(\bu^0_0,\bu^0_0) + \sum_{j=1}^J \eehform(\bu^0_j,\bu^0_j) \leq c \left(1 + \frac{H^4}{\delta^4} \right) \eehform(\bu^0,\bu^0),
\end{align*}
with $\bu^0_j = \curl s_j$, where $s_j\in S(\Omega,\mesh)$ is chosen as in the proof of \cref{lemma:decomposition_e_null}.

Further, by a classical result in domain decomposition theory that goes back to \cite{DryjaWidlund94}, every $r \in S(\Omega,\mesh)$
has a decomposition $r=\sum_{j=0}^J r_j$, such that
\begin{align*}
    \sum_{j=0}^J \Honenorm{r_j}^2 \leq c \left(1+\frac{H}{\delta}\right) \Honenorm{r}^2,
\end{align*}
where $\Honenorm{r} = \ltwonorm{\nabla r}$.
By choosing $\bv^0_j=\curl r_j$, we get for the second term in~\eqref{eq:globaldecompkernel}
\begin{gather} \label{eq:ltwo_decomp_hdivconf}
\begin{aligned}
    \sum_{j=0}^J \kform(\bv^0_j,\bv^0_j)
    &\leq \scaledRmax \sum_{j=0}^J \ltwonorm{\curl r_j}^2
    = \scaledRmax \sum_{j=0}^J \Honenorm{r_j}^2
    \\
    &\leq c \, \scaledRmax \left(1+\frac{H}{\delta}\right) \Honenorm{r}^2
    = c \, \scaledRmax \left(1+\frac{H}{\delta}\right) \ltwonorm{\curl r}^2
    \\
    &\leq c \, \frac{\scaledRmax}{\scaledRmin} \left(1+\frac{H}{\delta}\right) \kform(\bv^0,\bv^0)
\end{aligned}
\end{gather}
where we have used~\eqref{eq:normequivalence_k}. The assertion of the lemma follows then
by choosing $\bw^0_j = (\bu^0_j, \bv^0_j) = (\curl s_j, \curl r_j) \in\Wj[0]$ for all $j=0,\ldots,J$.
\end{proof}

We are ready to prove the main result of this section.

\begin{theorem}[Stable decomposition]\label{lemma:StableDecomp}
Every $\bw \in W_h$ admits a decomposition of the form $\bw = \sum_{j=0}^J \bw_j$
with $\bw_j\in \Wj$, which satisfies the bound
\begin{gather*}
    \sum_{j=0}^J \locsingularform(\bw_j,\bw_j)
    \leq \constK \left(1+\frac{H^4}{\delta^4}\right) \singularform(\bw,\bw),
\end{gather*}
where $\constK=c \max\set{\nfrac{\scaledRmax}{\scaledRmin},\scaledRmax}$ for some constant $c>0$
that is independent of the model parameters in the discrete bilinear form $\singularform(\cdot,\cdot)$ defined in~\eqref{eq:SPP-discretebilinearform} and independent of the number of subdomains $J$ as well as of the discretization parameters $h$ and $\tau$.
\end{theorem}

\begin{proof}
For $\bw = {\bm 0}$ the result is trivial. So let $\bw \in W_h$, $\bw \ne {\bm 0}$, be arbitrary.
For the sake of simplicity, we indicate dependencies of constants by subscripts.
To start with, we decompose $\bw$ in the form
\begin{align*}
\bw := \vv{\bu}{\bv} &= \vv{\bu^0}{\bv^0} + \vv{\bu^\perp}{\bv^\perp},
\end{align*}
with $(\bu^0,\bv^0)\in U^0\times V^0$ and $(\bu^\perp,\bv^\perp)\in U^\perp\times V^\perp$.
Since we need to control the term $\ddform(\bu+\bv,\bu+\bv)$, we 
define $\bphi = \bu^\perp + \bv^\perp$
and decompose it into $\bphi=\bphi^0+\bphi^\perp$, $\bphi^0\in V^0$, where $\bphi^0$ and $\bphi^\perp$ are orthogonal with respect to
$\kform(\cdot,\cdot)$.
Then, the decomposition of $\bw$ that we will use in our proof reads as
\begin{gather} \label{eq:decomposition_global}
    \begin{aligned}
    \bw = \vv{\bu}{\bv}
    &= \vv{\bu^0}{\bv^0} + \vv{\bu^\perp}{-\bu^\perp} + \vv{\bm 0}{\bphi}
    \\
    =& \vv{\bu^0}{\bv^0} + \vv{\bm 0}{\bphi^0} + \vv{\bu^\perp}{-\bu^\perp} + \vv{\bm 0}{\bphi^\perp}
    =: \bw^0+\bw^1+\bw^2+\bw^3.
    \end{aligned}
\end{gather}

\noindent
We decompose each component of each summand in~\eqref{eq:decomposition_global} according to
\begin{align*}
    &\sum_{j=0}^J \bu_j^0 = \bu^0,
    &&\sum_{j=0}^J \bv_j^0 = \bv^0,
    &&\sum_{j=0}^J \bphi_j^0 = \bphi^0, \\
    &\sum_{j=0}^J \bu_j^1 = \bu^\perp,
    &&\sum_{j=0}^J \bv_j^1 = \bv^\perp,
    &&\sum_{j=0}^J \bphi_j^1 = \bphi^\perp,
\end{align*}
where $\bu_j^0\in U_j^0$, $\bv_j^0\in V_j^0$, $\bphi_j^0\in V_j^0$, $\bu_j^1\in U_j$, $\bv_j^1\in V_j$ and $\bphi_j^1\in V_j$ are specified below. 
The superscript $1$ of $\bu_j^1$, $\bv_j^1$ and $\bphi_j^1$ indicates that these terms are not orthogonal to $\bu_j^0$, $\bv_j^0$ and $\bphi_j^0$ with respect to the inner products $\eehform(\cdot,\cdot)$, $\kform(\cdot,\cdot)$ and $\kform(\cdot,\cdot)$, respectively.
Now, we define $\bw_j:=\bw_j^0 + \bw_j^1 + \bw_j^2 + \bw_j^3$, where
\begin{align*}
    \bw_j^0 &= \vv{\bu_j^0}{\bv_j^0} \in W_j^0, \qquad& \bw_j^1 &= \vv{\bm 0}{\bphi_j^0} \in W_j^0, \\
    \bw_j^2 &= \vv{\bu_j^1}{-\bu_j^1} \in W_j, & \bw_j^3 &= \vv{\bm 0}{\bphi_j^1} \in W_j,
\end{align*}
and herewith estimate
\begin{gather}
\label{eq:local_decomp_a}
  \begin{aligned}
    \sum_{j=0}^J \locsingularform(\bw_j,\bw_j)
    &\leq c\sum_{j=0}^J \left(\locsingularform(\bw_j^0,\bw_j^0) + \locsingularform(\bw_j^1,\bw_j^1) \right.\\
    &\left.\hspace{4em} + \locsingularform(\bw_j^2,\bw_j^2) + \locsingularform(\bw_j^3,\bw_j^3) \right).
  \end{aligned}
\end{gather}
For the decomposition in the orthogonal complement $W_h^\perp$, we will use the stability estimates
\begin{align}
    \kform(\bphi^0,\bphi^0) &\leq \kform(\bphi,\bphi) = \kform(\bu^\perp+\bv^\perp,\bu^\perp+\bv^\perp),  \label{eq:stab-phinull}\\
    \kform(\bphi^\perp,\bphi^\perp) &\leq \kform(\bphi,\bphi) = \kform(\bu^\perp+\bv^\perp,\bu^\perp+\bv^\perp)  \label{eq:stab-phiperp}.
\end{align}
Further, we note that
\begin{gather}
    \ddform(\bphi^\perp,\bphi^\perp) = \ddform(\bphi,\bphi) = \ddform(\bu^\perp+\bv^\perp,\bu^\perp+\bv^\perp)
    \label{eq:equal-divphiperp}.
\end{gather}
Then, the stability of the $L^2$-decomposition \eqref{eq:ltwo_decomp_hdivconf} and the estimate \eqref{eq:stab-phinull} yield
\begin{gather} \label{eq:local_decomp_wone}
  \begin{aligned}
    \sum_{j=0}^J \locsingularform(\bw_j^1,\bw_j^1)
    = \sum_{j=0}^J & \kform(\bphi_j^0,\bphi_j^0)
    \leq c_{H,\delta,\scaledpermeability} \kform(\bphi^0,\bphi^0)
    \\
    &\leq c_{H,\delta,\scaledpermeability} \left(\kform(\bu^\perp,\bu^\perp) + \kform(\bv^\perp,\bv^\perp) \right).
  \end{aligned}
\end{gather}
Moreover, by~\cref{lemma:decomposition_e_perp} and \cite{ToselliWidlund10,ArnoldFalkWinther97Hdiv}, we have
\begin{gather} \label{eq:local_decomp_wtwo}
  \begin{aligned}
    \sum_{j=0}^J \locsingularform(\bw_j^2,\bw_j^2)
    &= \sum_{j=0}^J \left(\eejform(\bu_j^1,\bu_j^1) + \kform(\bu_j^1,\bu_j^1) + \scaledlambda \ddform(\bu_j^1,\bu_j^1)\right)
    \\
    &\leq c_{H,\delta} \left(\eehform(\bu^\perp,\bu^\perp) + \kform(\bu^\perp,\bu^\perp) + \scaledlambda \ddform(\bu^\perp,\bu^\perp) \right).
  \end{aligned}
\end{gather}
Again, by~\cite{ToselliWidlund10,ArnoldFalkWinther97Hdiv}, stability estimate \eqref{eq:stab-phiperp}, and equality~\eqref{eq:equal-divphiperp}, 
we obtain
\begin{align}
    \sum_{j=0}^J \locsingularform(\bw_j^3,\bw_j^3)
    &= \sum_{j=0}^J \left( \kform(\bphi_j^1,\bphi_j^1) + \scaledstorage^{-1} \ddform(\bphi_j^1,\bphi_j^1) \right) \nonumber
    \\
    &\leq c_{H,\delta} \left( \kform(\bphi^\perp,\bphi^\perp) + \scaledstorage^{-1} \ddform(\bphi^\perp,\bphi^\perp) \right) \label{eq:local_decomp_wthree}
    \\
    &\leq c_{H,\delta} \left( \kform(\bu^\perp,\bu^\perp) + \kform(\bv^\perp,\bv^\perp) + \scaledstorage^{-1} \ddform(\bu^\perp+\bv^\perp,\bu^\perp+\bv^\perp) \right). \nonumber
\end{align}
Due to the special choice of the decomposition~\eqref{eq:decomposition_global} the term $\kform(\bu^\perp,\bu^\perp)$ appears in~\eqref{eq:local_decomp_wthree},
which we further estimate using the Poincar\'e and Korn inequalities
\begin{gather}
    \label{eq:poincare}
    \kform(\bu^\perp,\bu^\perp) \leq c_{\Omega} \scaledRmax \eehform(\bu^\perp,\bu^\perp).
\end{gather}

Thus, by \eqref{eq:local_decomp_a}, \cref{lemma:decomposition_kernel} and collecting the estimates
\eqref{eq:local_decomp_wone}, \eqref{eq:local_decomp_wtwo}, \eqref{eq:local_decomp_wthree}, and
\eqref{eq:poincare}, we obtain with $\bw^\perp=(\bu^\perp,\bv^\perp)^T\in W_h^\perp$
\begin{align*}
    \sum_{j=0}^J &\locsingularform(\bw_j,\bw_j)
    \\
    \leq& c_{H,\delta,\scaledpermeability} \Bigl( \eehform(\bu^0,\bu^0) + \kform(\bv^0,\bv^0)
    \\
    &\quad + \eehform(\bu^\perp,\bu^\perp) + \kform(\bu^\perp,\bu^\perp) + \scaledlambda \ddform(\bu^\perp,\bu^\perp)
    + \scaledstorage^{-1} \ddform(\bu^\perp+\bv^\perp,\bu^\perp+\bv^\perp) \Bigr)
    \\
    \leq& c_{H,\delta,\scaledpermeability,\Omega} \Bigl(
    \singularform(\bw^0,\bw^0) + \singularform(\bw^\perp,\bw^\perp)
    \Bigr)
    \\
    =&  c_{H,\delta,\scaledpermeability,\Omega} \singularform(\bw,\bw)
\end{align*}
where the constant $c_{H,\delta,\scaledpermeability,\Omega}$ for a proper overlap $\delta$ depends only on~$\Omega$ and the bounds $\scaledRmin$ and $\scaledRmax$.
\end{proof}

\subsection{Results for the singularly perturbed problem}
\label{section:SPP-convergence}
We now prove convergence of the Schwarz methods for the singularly perturbed problem by applying the theory for symmetric positive definite problems as in~\cite{ToselliWidlund10}.
Therefore, we further establish a local stability estimate in \cref{lemma:SPP-localstability} and strengthened Cauchy-Schwarz inequalities in~\eqref{eq:SPP-strengthenedCS}, which together with the stable decomposition in \cref{lemma:StableDecomp} and  \cref{assumption:finite-covering,assumption:sufficient-overlap} on the covering of the domain prove the results of this section, i.e., \cref{theorem:SPP-additive,theorem:SPP-multiplicative}.

Local stability, according to~\cite{ToselliWidlund10}, means that the estimates
\begin{gather*}
    \singularform(\bw_j,\bw_j) \le \localstabconstant \locsingularform(\bw_j,\bw_j) \qquad\forall \bw_j\in\Wj
\end{gather*}
hold true for all $0\le j\le J$ with a constant $\localstabconstant>0$.
Since we have chosen exact bilinear forms on the local spaces, i.e. $\locsingularform(\cdot,\cdot) = \singularform(\cdot,\cdot)$ restricted to $\Wj$,
local stability is trivially fulfilled with $\localstabconstant=1$ for all $1\le j\le J$.
Thus, it remains to consider the case $j=0$.

\begin{lemma}[Local stability]\label{lemma:SPP-localstability}
For each $\bw_0\in W_0$ it holds
\begin{gather*}
    \singularform(\bw_0,\bw_0) \le \frac{H}{h} \locsingularform[0](\bw_0,\bw_0).
\end{gather*}
\end{lemma}

\begin{proof}
On the coarse space, we have chosen a non-inherited form for the approximation
of the elasticity bilinear form, namely $\eeHform(\cdot,\cdot)$ instead of $\eehform(\cdot,\cdot)$, which differs in the face
and boundary terms because of the different cell size $H\ge h$. Due to the continuity of
the coarse displacement functions $\bu_0 = \bu_H \in U_0$ on every coarse cell, the jump terms $\ltwonorm[\face]{\jump{\bu_H}}$
vanish for all faces $\face$ that lie in the interior of a coarse cell. Therefore, we have
\begin{gather}\label{eq:localstab-eehform}
    \eehform(\bu_0,\bu_0) \leq \frac{H}{h} \eeHform(\bu_0,\bu_0).
\end{gather}
Since all other terms of the coarse space operator $\locsingularform[0](\cdot,\cdot)$
are chosen to be exact, the statement of \cref{lemma:SPP-localstability} follows by estimate~\eqref{eq:localstab-eehform}.
\end{proof}

\begin{remark}\label{remark:relaxation-multiplicative}
The convergence theory for multiplicative Schwarz methods in \cite{ToselliWidlund10} requires that $0<\localstabconstant<2$.
In a typical application scenario, the coarse mesh size might be $H=2h$ resulting in a factor $\nfrac{H}{h}=2$ which violates the above condition.
Hence, an additional relaxation factor $\relaxationconstant<1$ has to be introduced in the definition of the coarse bilinear form, i.e.,
$\locsingularform[0](\cdot,\cdot)=\relaxationconstant\coarsesingularform(\cdot,\cdot)$,
that scales $\coarsesingularform(\cdot,\cdot)$ such that $\relaxationconstant\nfrac{H}{h}<2$.
However, we would like to point out that this scaling is not required in our experiments and we choose $\relaxationconstant=1$ in these cases.
\end{remark}

By \cref{assumption:sufficient-overlap} and standard arguments, see \cite{ToselliWidlund10}, there hold the strengthened Cauchy-Schwarz inequalities
\begin{gather}\label{eq:SPP-strengthenedCS}
    \abs{\singularform(\bw_i,\bw_j)}
    \leq \CSconstant \singularform(\bw_i,\bw_i)^{\frac12} \singularform(\bw_j,\bw_j)^{\frac12}
\end{gather}
for $\bw_i\in W_i$, $\bw_j\in\Wj$, $1\le i,j\le J$, and constants $0\leq\CSconstant\leq1$.
Alongside \cite{ToselliWidlund10} we denote the spectral radius of $\mathcal{E} = \left\{\CSconstant\right\}$ by $\spectralradius$.

The maximal and minimal eigenvalues of $\singularPad \singularoperator$ are characterized by the Rayleigh quotients
\begin{align*}
    \eigenvalue_{\max}(\singularPad \singularoperator)
    &= \sup_{\bw\in W_h} \frac{\singularform(\singularPad \singularoperator \bw,\bw)}{\singularform(\bw,\bw)},
    \\
    \eigenvalue_{\min}(\singularPad \singularoperator)
    &= \inf_{\bw\in W_h} \frac{\singularform(\singularPad \singularoperator \bw,\bw)}{\singularform(\bw,\bw)}.
\end{align*}
We are now ready to state and prove the main results of this section, \cref{theorem:SPP-additive,theorem:SPP-multiplicative}.

\begin{theorem}
\label{theorem:SPP-additive}
The two-level additive Schwarz preconditioner~\eqref{eq:SPP-additiveschwarz} for the singularly perturbed system~\eqref{eq:singularproblem} satisfies
\begin{gather*}
    \eigenvalue_{\min}(\singularPad \singularoperator)
    \ge \left(\constK \left(1+\frac{H^4}{\delta^4}\right)\right)^{-1},
    \qquad
    \eigenvalue_{\max}(\singularPad \singularoperator)
    \le \frac{H}{h} \left(\spectralradius + 1 \right),
\end{gather*}
where the constant $\constK$ depends on the bounds $\scaledRmin$ and $\scaledRmax$ from~\eqref{eq:normequivalence_k}, but is independent of the mesh size $h$ and the other model parameters $\scaledlambda$, and $\scaledstorage$.
In particular, we obtain a uniform preconditioner under the assumption that $H \le c h$ for some constant $c$.
\end{theorem}

\begin{proof}
  By \cref{lemma:biot:posdef} we are able to apply the Schwarz convergence theory for symmetric positive definite problems, cf. \cite[Chapter 2]{ToselliWidlund10}, to the singularly perturbed problem.
  In particular, the stable decomposition from \cref{lemma:StableDecomp}, the local stability estimate from \cref{lemma:SPP-localstability}, and the strengthened Cauchy-Schwarz inequalities~\eqref{eq:SPP-strengthenedCS} constitute the assumptions of the abstract condition number estimate in~\cite[Theorem 2.7]{ToselliWidlund10}. Hence, we can apply this result which concludes the proof.
\end{proof}

\begin{theorem}
\label{theorem:SPP-multiplicative}
Under the same assumptions as in \cref{theorem:SPP-additive}, the multiplicative Schwarz method~\eqref{eq:SPP-multiplicativeschwarz} for the singularly perturbed problem converges with a contraction number depending on the bounds $\scaledRmin$ and $\scaledRmax$ from equation~\eqref{eq:normequivalence_k}, but independent of the mesh size $h$ and the other model parameters $\scaledlambda$, and $\scaledstorage$.
\end{theorem}

\begin{proof}
  We invoke~\cite[Theorem 2.9]{ToselliWidlund10}, where again the stable decomposition from \cref{lemma:StableDecomp}, the local stability estimate from \cref{lemma:SPP-localstability}, and the strengthened Cauchy-Schwarz inequalities~\eqref{eq:SPP-strengthenedCS} constitute the assumptions. For the coarse grid space, we additionally refer to \cref{remark:relaxation-multiplicative}.
\end{proof}

\subsection{Proofs of the main results}
\label{sec:proof}

\begin{proof}[Proof of \cref{theorem:additive}]
  \Cref{theorem:SPP-additive} guarantees the boundedness of the spectrum of $\singularPad \singularoperator$ with respect to the singularly perturbed problem, uniform with respect to the perturbation parameter $\scaledstorage^{-1}$ for any $\scaledstorage >0$.
  Thus, the equivalence of the three-field formulation and the singularly perturbed problem yields the convergence proof.
\end{proof}

\begin{proof}[Proof of \cref{theorem:multiplicative}]
  \Cref{theorem:SPP-multiplicative} guarantees the convergence of the multiplicative Schwarz method~\eqref{eq:SPP-multiplicativeschwarz} with respect to the singularly perturbed problem with a contraction number independent of the perturbation parameter $\scaledstorage^{-1}$, and as a consequence, uniform for any $\scaledstorage >0$.
  Thus, the equivalence of the three-field formulation and the singularly perturbed problem yields the convergence proof.
\end{proof}
\section{Numerical experiments}
\label{sec:experiments}

For testing the numerical performance of the two-level Schwarz
preconditioners defined in \cref{sec:twolevelalgorithm}, we follow the
setting in \cite[Section 6.1]{HongKraus19} and construct a test case
with solution in
$[H^1_0(\Omega)]^2\times\Hdiv_0(\Omega)\times L^2_0(\Omega)$.
For a more in depth investigation with different kinds of boundary
conditions, as well as a comparison to multilevel algorithms we refer the
reader to~\cite{Meggendorfer23thesis}.
The implementation of the solvers is based on the finite element library \texttt{DEAL.II}~\cite{dealii22}.

Consider \eqref{eq:modelproblem} with isotropic permeability
$\scaledpermeability=\scaledkappa I$ and right hand side given by
\begin{gather*}
  \bbf =
  \begin{pmatrix}
    900 \partial_x\phi - \partial_y^3\phi - \partial_x^2\partial_y\phi\\
    900 \partial_y\phi + \partial_x^3\phi + \partial_x\partial_y^2\phi
  \end{pmatrix}, \qquad
  g = 900 \scaledkappa\upDelta\phi - \scaledstorage(900\phi - 1),
\end{gather*}
where $\phi$ is defined on the square $\Omega=(0,1)\times(0,1)$ by
\begin{gather*}
  \phi = x^2(x-1)^2y^2(y-1)^2.
\end{gather*}
Homogeneous Dirichlet boundary conditions are prescribed for the solid
displacement $\bu$ as well as for the normal direction of the
seepage velocity $\bv$ on the whole boundary $\partial\Omega$, i.e.,
\begin{gather*}
  \bu = \bm 0 \quad \text{on } \partial\Omega,\qquad
  \bv\cdot \bn = 0 \quad \text{on } \partial\Omega.
\end{gather*}
Here, $\bn$ denotes the unit outward normal vector. The solution is
defined only up to an additive constant for the pressure. Thus, we
search for a mean-value free solution in the pressure component
satisfying
\begin{gather*}
  \int_\Omega p\,dx=0.
\end{gather*}
Then, the unique solution to this system is given by
\begin{gather*}
  \bu =
  \begin{pmatrix}
    \partial_y \phi\\ -\partial_x \phi
  \end{pmatrix}, \qquad
  p = 900 \phi - 1, \qquad
  \bv = -900\scaledkappa\nabla p.
\end{gather*}

The mesh consists of an equidistant subdivision of $\Omega\subset\R^2$ into rectangular cells.
The system is discretized with triplets $RT_k\times RT_k\times Q_k$, $k\ge 0$, of equal order Raviart-Thomas and discontinuous cell-wise polynomial finite element functions, such that the matching condition
\begin{gather*}
    \div RT_k = Q_k
\end{gather*}
is fulfilled.
To solve the algebraic equations we use the Krylov subspace method GMRES preconditioned by the multiplicative two-level Schwarz method \eqref{eq:multiplicativeschwarz}, which combines a multiplicative subdomain solver with a multiplicative application of the coarse grid solver.
As stopping criterion we use a reduction of the unpreconditioned starting residual
measured in the $L^2$-norm by a factor $10^{8}$.
The subdomains $\Omega_j$ form vertex patches, which are unions of four cells for all interior vertices with an overlap of size $\delta=h$.
Patches consisting of two cells for vertices at the boundary $\partial\Omega$ and single cells for vertices in corners are excluded from the set of patches, since we could not observe any beneficial effects of their inclusion for calculations with homogeneous boundary conditions, cf.~\cite[Section 3.5]{Meggendorfer23thesis}.
To realize the homogeneous boundary conditions on the boundary $\partial\Omega_j$ of the vertex patches, we build the patch matrices only out of interior degrees of freedom of each vertex patch $\Omega_j$.
The coarse space is always assembled one level below the actual level, i.e., $H=2h$ leading to a constant factor $\frac{H}{\delta}=2$.

The coarse and local problems are solved using a singular value decomposition in order to deal properly with the one-dimensional kernel of their operators caused by the nonuniqueness of the pressure solution on each subdomain.

\begin{table}[tp]
  \centering
  \begin{tabular}[t]{c|cccc}
    \toprule
    $\scaledlambda$     &$1$&$10^6$&$1$   &$10^6$\\
    $\scaledkappa^{-1}$ &$1$&$1$   &$10^6$&$10^6$\\
    $h$ &&&& \\
    \hline
    $\nfrac14$     &4&5&6&6\\
    $\nfrac18$     &5&6&6&6\\
    $\nfrac1{16}$  &5&6&6&6\\
    $\nfrac1{32}$  &4&6&6&6\\
    \bottomrule
  \end{tabular}
  \hspace{2em}
  \begin{tabular}[t]{c|cccc}
    \toprule
    $\scaledlambda$     &$1$&$10^6$&$1$   &$10^6$ \\
    $\scaledkappa^{-1}$ &$1$&$1$   &$10^6$&$10^6$ \\
    $h$ & \\
    \hline
    $\nfrac1{16}$ &5&7&7&7\\
    $\nfrac1{32}$ &5&8&8&8\\
    $\nfrac1{64}$ &5&8&8&8\\
    $\nfrac1{128}$&5&8&8&8\\
    \bottomrule
  \end{tabular}
  \caption{Uniformity of iteration counts of GMRES for different parameter settings when $\scaledstorage=0$, $RT_2\times RT_2\times Q_2$.
    Multiplicative two-level Schwarz (left table), multiplicative multilevel Schwarz (right table).}
  \label{tab:uniformity}
\end{table}%

In~\cref{tab:uniformity} (left side) we observe uniform convergence of the method in a set of sample calculations with $RT_2\times RT_2\times Q_2$ finite elements, from which we conclude that no additional relaxation of the coarse grid bilinear form is needed albeit required by the theory. This observation gave reason to~\cref{remark:relaxation-multiplicative}.
Furthermore, throughout all calculations with polynomial degree $k=2$, we found that the iteration counts remain constant for a mesh size below $h=\frac{1}{32}$.
This holds true for all variations of the model parameters in the system, i.e., $\scaledlambda$, $\scaledkappa^{-1}$, and $\scaledstorage$.
A comparison calculation in~\cref{tab:uniformity} (right-hand side) also confirms this for more refined meshes with the multilevel Schwarz analogue $\Mmu$ of the two-level algorithm \eqref{eq:multiplicativeschwarz}.
For a definition of $\Mmu$ let there be given a hierarchy of nested meshes $\mesh[h_0]\sqsubset\cdots\sqsubset\mesh[h_L]$ with corresponding finite element spaces $X_{h_0}\subset\cdots\subset X_{h_L}$ for mesh sizes $h_0\le\ldots\le h_L=h$.
Here, the symbol $\sqsubset$ means that mesh $\mesh[h_{\ell+1}]$ results by a refinement of $\mesh[h_{\ell}]$.
Then, $\Mmu$ is defined by
\begin{align*}
    \Mmu &= \left(I - \Emu^{\text{MG}}\right) \systemoperator^{-1},
    \ \
    \Emu^{\text{MG}}
    = \prod_{j=1}^{J_L} \left(I - \mixedP_j^L\right) \cdots \prod_{j=1}^{J_1} \left(I - \mixedP_j^1\right) \left(I - \mixedP_0^0\right),
\end{align*}
where $J_\ell$ denotes the number of subdomains on each level $\ell=0,\ldots,L$, and $\mixedP_j^\ell$ are the projection operators of each subdomain on each level.
In the multilevel setting we observe about two iterations more than in the two-level case, which is due to the relaxed coarse scale solve.
Nevertheless, the numbers show the independence of the iteration counts with respect to the levels on a mesh with $h=\frac{1}{32}$ already.

\begin{table}[tp]
  \centering
  \begin{tabular}[t]{cc|ccccccc}
    \toprule
    & & \multicolumn{7}{c}{$\scaledkappa^{-1}$}\\
    & & $10^{-6}$ & $10^{-4}$ & $10^{-2}$ & $1$ & $10^{2}$ & $10^{4}$ & $10^{6}$
    \\
    \hline
    \multirow{8}{*}{$\scaledlambda$}
    & $10^{-6}$ &2&3&3&4&4&5&6\\
    & $10^{-4}$ &2&3&3&4&4&5&6\\
    & $10^{-2}$ &2&3&3&4&4&5&6\\
    & $1$       &2&3&4&4&4&5&6\\
    & $10^2$    &2&3&4&5&6&6&6\\
    & $10^4$    &2&2&3&5&6&6&6\\
    & $10^6$    &2&2&3&4&6&6&6\\
    \bottomrule
  \end{tabular}
  \caption{Robustness of iteration counts of GMRES with respect to $\scaledlambda$ and $\scaledkappa^{-1}$ (for $\scaledstorage=0$).
    Multiplicative two-level Schwarz.
    $RT_2\times RT_2\times Q_2$, $h=\frac1{32}$.}
  \label{tab:mult_lambda_resistance_fe2}
\end{table}%

\begin{table}[tp]
  \centering
  \begin{tabular}[t]{c|ccccccccc}
    \toprule
    $\scaledstorage$ & $0$ & $10^{-10}$ & $10^{-4}$ & $10^{-2}$ & $1$ & $10^2$ & $10^4$ & $10^6$ & $10^8$ \\
    \hline
    &4&4&4&4&4&4&6&2&1\\
    \bottomrule
  \end{tabular}
  \caption{Robustness of iteration counts of GMRES with respect to $\scaledstorage$ (for $\scaledlambda=\scaledkappa^{-1}=1$).
    Multiplicative two-level Schwarz.
    $RT_2\times RT_2\times Q_2$, $h=\frac1{32}$.}
  \label{tab:mult_storage_fe2}
\end{table}%

Now, we turn our attention to an examination of the robustness of the method with respect to the scaled model parameters in the system, i.e., $\scaledlambda$, $\scaledkappa^{-1}$, and $\scaledstorage$.
In this test scenario we choose again polynomial degree $k=2$ and show the converged iteration counts for a fixed mesh with mesh size $h=\frac1{32}$.
The robustness of the multiplicative two-level Schwarz method with respect to varying the scaled parameters $\scaledlambda$ and $\scaledkappa^{-1}$ is demonstrated in \cref{tab:mult_lambda_resistance_fe2}.
Here we have set the scaled storage capacity coefficient $\scaledstorage=0$, which actually represents the most difficult case of the saddle point problem for the solver.
The iteration counts of GMRES are at most six for the whole range of $10^{-6}\le\scaledlambda\le 10^6$, and $10^{-6}\le\scaledkappa^{-1}\le 10^6$.
For small $\scaledkappa^{-1}$ the iteration numbers are even smaller, which indicates the dependence
on $\scaledRmax$ that we have seen in the convergence proof in \cref{sec:analysis}.
On the other hand, we do not observe an increase in iterations for large $\scaledkappa^{-1}$, which corresponds to small permeabilities.
The situation that $\scaledkappa^{-1}$ takes small values is also realistic and occurs if the product of $\mu$ and $\tau$ and the unscaled permeability is large,
cf. the scaling of the physical parameters in \eqref{eq:rescaledparameters}.
Moreover, the scaled parameter $\scaledlambda$ is defined as the ratio of the Lam\'e coefficients that might be small in certain cases or become large for incompressible materials.
For variations in $\scaledlambda$, the table shows nearly constant iterations for each fixed $\scaledkappa^{-1}$; only for $\scaledkappa^{-1}=10^2$ there is a slight shift in the numbers when $\scaledlambda>10^2$.

Variations in the scaled storage coefficient $\scaledstorage$ do not effect the performance of the method for values $\scaledstorage\le 10^2$ as can be seen in~\cref{tab:mult_storage_fe2}.
Only for very large values of $\scaledstorage > 10^2$ the iterations increase very slightly before they drop to 1.
This result is presented for $\scaledlambda=\scaledkappa^{-1}=1$ only but holds true for all possible combinations of $\scaledlambda$ and $\scaledkappa^{-1}$.
Note that in realistic calculations the scaled storage coefficient $\scaledstorage$ might take values even larger than one if the poroelastic medium contains soft inclusions modeled by a storage capacity greater than zero, together with small values of the Biot-Willis constant $\alpha$ and/or a large Lam\'e coefficient $\mu$, cf. \eqref{eq:rescaledparameters}.

\begin{table}[tp]
  \centering
  \begin{tabular}[t]{c|ccccccc}
    \toprule
    & \multicolumn{7}{c}{$\scaledkappa^{-1}$} \\
    $h$ & $10^{-6}$ & $10^{-4}$ & $10^{-2}$ & $1$ & $10^{2}$ & $10^{4}$ & $10^{6}$\\
    \hline
    $\nfrac18$     & 3 & 4 & 6 & 8 & 9 & 11 & 12\\
    $\nfrac1{16}$  & 3 & 5 & 7 & 8 & 9 & 13 & 14\\
    $\nfrac1{32}$  & 3 & 5 & 7 & 8 & 8 & 11 & 15\\
    $\nfrac1{64}$  & 3 & 5 & 7 & 8 & 8 &  9 & 15\\
    $\nfrac1{128}$ & 4 & 6 & 7 & 8 & 8 &  9 & 15\\
    \bottomrule
  \end{tabular}
  \caption{Lowest order case $RT_0\times RT_0\times Q_0$.
    Iteration counts of GMRES with respect to $\scaledkappa^{-1}$.
    Multiplicative two-level Schwarz.
    $\scaledlambda=100$, $\scaledstorage=0$.
    }
  \label{tab:mult_lambda_resistance_fe0}
\end{table}%

Last, we turn our attention to the lower order case for polynomial degrees $k=0$ and $k=1$.
In this situation, we observed some sort of stiffness of the method compared to polynomial degrees $k \geq 2$.
In a sample calculation in the lowest-order case $k=0$ in \cref{tab:mult_lambda_resistance_fe0} with $\scaledlambda = 100$, $\scaledkappa^{-1} = 1$, and $\scaledstorage = 0$ we need more mesh refinements for some parameter settings (down to $h=\frac{1}{128}$) until the iteration
counts remain unchanged.
And we notice a stronger influence of the variation in the permeability on the iteration counts as they grow faster for larger values of $\scaledkappa^{-1}$.
Nevertheless, the method shows good performance and reasonable robustness also in the lowest-order case, i.e., employing the $RT_0\times RT_0\times Q_0$ space triplet.

\section{Conclusion}
In this article, we proposed two-level Schwarz methods to solve $\Hdiv$-conforming discretizations of Biot's quasi-static consolidation model.
We rigorously proved convergence of the monolithic algorithms by transforming the three-field formulation into an equivalent singularly perturbed symmetric positive definite system, for which a stable decomposition and local stability bounds form the basis for applying the abstract Schwarz convergence theory for noninherited forms.
Finally, independence with respect to the perturbation parameter resulted in the proof of convergence, robust with respect to the model parameters under the assumption that the permeability tensor can be bounded in the sense of~\eqref{eq:normequivalence_k}.
The numerical tests demonstrated the performance and robustness of the multiplicative two-level Schwarz method for different polynomial degrees of the mixed finite element spaces and confirmed the theoretical results, even for nearly incompressible and low-permeable materials.
Beyond the theoretical results, a multilevel Schwarz approach showed comparable performance at lower cost.

The three-dimensional case as well as the investigation of high-frequency-high-contrast problems are the subjects of future work.

\section*{Acknowledgment}
The third author acknowledges the funding by the German Science Fund (Deut\-sche Forschungsgemeinschaft, DFG) as part of the project “Physics-oriented solvers for multicompartmental poromechanics” under grant number 456235063.

\bibliographystyle{alpha}
\bibliography{bib}

\end{document}